\documentclass[12pt]{amsart}
\usepackage{fullpage}
\usepackage{amscd,amsmath,amsthm,amssymb}
\usepackage{comment}
\usepackage[all]{xy}
\usepackage{color}

\usepackage[lined,algonl,boxed,norelsize]{algorithm2e}
\let\chapter\undefined

\usepackage[
backref,
pdfauthor={  }, % add other authors
]{hyperref}
\usepackage[
alphabetic,
backrefs,
msc-links,
nobysame,
lite,
]{amsrefs} % for bibliography
\usepackage{tikz, float} \usetikzlibrary {positioning}
\def\A{{\mathcal A}}

\def\opn#1#2{\def#1{\operatorname{#2}}} % to make operators

\opn\set{set}
\opn\supp{supp}
\opn\ch{ch}
\opn\depth{depth} 
\opn\codim{codim}
\opn\ini{in} 
\opn\LM{LM}
\opn\LC{LC}
\opn\NF{NF}
\opn\Merge{Merge}
\opn\sgn{sgn}
\opn\div{div} 
\opn\Div{Div} 
\opn\Pic{Pic}
\opn\Prin{Prin}
\opn\Del{Del}
\opn\op{op}
\opn\indeg{indeg} 
\opn\outdeg{outdeg}
\opn\red{red}
\opn\Spec{Spec} 
\opn\Supp{Supp} 
\opn\supp{supp} 
\opn\Ker{Ker} 
\opn\Coker{Coker} 
\opn\Hom{Hom}
\opn\Tor{Tor} 
\opn\id{id}
\opn\span{span}
\opn\Image{Image}
\opn\con{conv} 
\opn\relint{rel.int} 
\opn\vol{vol}
\opn\syz{{\rm syz}}
\opn\spoly{{\rm spoly}}
\opn\LM{{\rm LM}}
\opn\lm{{\rm lm}}
\opn\lcm{{\rm lcm}} 
\opn\A{\mathcal A}
\opn\dist{dist}
\opn\pd{pd}
\opn\en{en}
\def\Implies{\ifmmode\Longrightarrow \else
        \unskip${}\Longrightarrow{}$\ignorespaces\fi}
\def\implies{\ifmmode\Rightarrow \else
        \unskip${}\Rightarrow{}$\ignorespaces\fi}
\def\iff{\ifmmode\Longleftrightarrow \else
        \unskip${}\Longleftrightarrow{}$\ignorespaces\fi}
\newtheorem{Theorem}{Theorem}[section]
\newtheorem{Lemma}[Theorem]{Lemma}
\newtheorem{Corollary}[Theorem]{Corollary}

\theoremstyle{remark}
\newtheorem{Remark}[Theorem]{Remark}

\theoremstyle{definition}
\newtheorem{Example}[Theorem]{Example}

\newtheorem{Definition}[Theorem]{Definition}
\newtheorem*{Notation}{Notation}
\def\qed{\ifhmode\textqed\fi
      \ifmmode\ifinner\quad\qedsymbol\else\dispqed\fi\fi}
\def\textqed{\unskip\nobreak\penalty50
       \hskip2em\hbox{}\nobreak\hfil\qedsymbol
       \parfillskip=0pt \finalhyphendemerits=0}
\def\dispqed{\rlap{\qquad\qedsymbol}}

\begin{document}

%%%%%%%%%%%%%%%%%%%%%%%%%%%%%%%%%%%%%%%%%%%%%%%%%%%%
%%%%%%%%%%%%%% TITLE! %%%%%%%%%%%%%%%%%%%%%%%%%%%%%%%%
%%%%%%%%%%%%%%%%%%%%%%%%%%%%%%%%%%%%%%%%%%%%%%%%%%%%

\title {Cellular resolutions from mapping cones}

\author {Anton Dochtermann}
\address{Department of Mathematics, University of Miami, USA
}
\email{anton@math.miami.edu}

\author {Fatemeh Mohammadi}
\address{Institut f\"ur Mathematik\\
Universit\"at Osnabr\"uck \\
%49069 Osnabr\"uck, 
 Germany}
\email{fatemeh@msri.org}

%\thanks{Fatemeh Mohammadi was supported by  the Alexander von Humboldt Foundation....}

\subjclass[2010]{}

\date{\today}
%%%%%%%%%%%%%%%%%%%%%%%%%%%%%%%%%%%%%%%%%%%%%%%%%%%%%%%%%%%%%%%%%%%%%%%%%%%%%%%%%%%%%%%%%%%%%%%%%%%%%%%%%%%%%%%
%%%%%%%%%%%%%%%%%%%%%%%%%%%%%%%%%%%%%%%%%%%%%%%%%%%%%%%%%%%%%%%%%%%%%%%%%%%%%%%%%%%%%%%%%%%%%%%%%%%%%%%%%%%%%%%

\begin{abstract}
One can iteratively obtain a free resolution of any monomial ideal $I$ by considering the mapping cone of the map of complexes associated
to adding one generator at a time.  Herzog and Takayama have shown that this
procedure yields a minimal resolution if $I$ has linear quotients, in which case
the mapping cone in each step cones a Koszul complex onto the previously constructed resolution.  Here we consider cellular realizations of these resolutions.  Extending a construction of Mermin we describe a regular CW-complex that supports the
resolutions of Herzog and Takayama in the case that $I$ has a `regular decomposition
function'.  By varying the choice of chain map we recover other known cellular
resolutions, including the `box of complexes' resolutions of Corso, Nagel, and
Reiner and the related `homomorphism complex' resolutions of Dochtermann and
Engstr\"om.  Other choices yield combinatorially distinct complexes with interesting
structure, and suggests a notion of a `space of cellular resolutions'.
\end{abstract}
%%%%%%%%%%%%%%%%%%%%%%%%%%%%%%%%%%%%%%%%%%%%%%%%%%%%%%%%%%%%%%%%%%%%%%%%%%%%%%%%%%%%%%%%%%%%%%%%%%%%%%%%%%%%%%%
%%%%%%%%%%%%%%%%%%%%%%%%%%%%%%%%%%%%%%%%%%%%%%%%%%%%%%%%%%%%%%%%%%%%%%%%%%%%%%%%%%%%%%%%%%%%%%%%%%%%%%%%%%%%%%%

\maketitle

%\tableofcontents

%%%%%%%%%%%%%%%%%%%%%%%%%%%%%%%%%%%%%%%%%%%%%%%%%%%%%%%%%%%%%%%%%%%%%%%%%%%%%%%%%%%%%%%%%%%%%%%%%%%%%%%%%%%%%%%
%%%%%%%%%%%%%%%%%%%%%%%%%%%%%%%%%%%%%%%%%%%%%%%%%%%%%%%%%%%%%%%%%%%%%%%%%%%%%%%%%%%%%%%%%%%%%%%%%%%%%%%%%%%%%%%

%%%%%%%%%%%%%%%%%%%%%

\section{Introduction}
Let $R = k[x_1, x_2, \dots, x_n]$ be the graded polynomial ring over a field $k$, and let $I \subset R$ be an ideal generated by monomials $I = 
\langle m_1, m_2, \dots, m_k \rangle$.  A popular game in combinatorial commutative algebra is to understand minimal graded resolutions of the monomial ideal $I$ under various restrictions.  In \cite{HT} Herzog and Takayama describe a general `mapping cone' procedure for constructing a free resolution of a monomial ideal $I$.  The basic idea is to utilize the short exact sequences that arise from adding one generator of $I$ at a time, and to iteratively build the resolution as a mapping cone of an appropriate map between previously constructed complexes.

\medskip

If the ideal $I$ has \emph{linear quotients} with respect to some ordering of its generators, the complexes and maps in question are particularly well-behaved, and this case the mapping cone construction leads to a \emph{minimal} resolution of $I$.  The class of ideals with linear quotients includes stable ideals, squarefree stable ideals, as well as matroidal ideals (relevant definitions below).  Stable monomial ideals themselves have a well-known minimal resolution first described in \cite{EK}, and the mapping cone construction generalizes this so-called \emph{Eliahou-Kervaire (EK) resolution}.  In this context there is a natural choice of homogeneous basis for each free module in the resolution of $I$.  In the case that the ideal $I$ has linear quotients, and furthermore has a \emph{regular decomposition function}, the differentials in the free resolution can also be described explicitly.  We review these concepts in Section \ref{Mapcones}.

\medskip

One way to describe the resolution of a (monomial) ideal is through the construction of a CW-complex whose vertices index the generators of $I$, and whose higher dimensional faces index the syzygies.  These so-called \emph{cellular resolutions} were first introduced by Bayer and Sturmfels in 
\cite{BayerSturmfels}.  A natural question to ask is if the mapping cone construction can be realized cellularly.  In \cite{Mermin} Mermin shows that the EK resolution is indeed cellular, and in the case that $I$ is generated in a fixed degree $d$, is in fact supported on a subcomplex of a suitably subdivided dilated simplex.

\medskip

In this paper we combine methods from \cite{HT} and \cite{Mermin} to show that one can obtain a larger class of cellular resolutions via the mapping cone construction.  Although the mapping cones of \cite{HT} are purely algebraic objects, they have a geometric interpretation which, in the context of ideals with linear quotients, we seek to combine with the decomposition function approach from \cite{HT}.  The basic idea is that if $I$ is an ideal with linear quotients then the mapping cone construction can be realized as an iterative procedure where in each step a geometric simplex is glued in a certain way onto the existing cellular resolution.

\medskip

In Section \ref{EKres} we extend Mermin's result to the case of ideals with linear quotients.  Our main result from that section is the following\footnote{While preparing this paper we learned of the recent preprint \cite{Afshin} of Goodarzi, where similar results were independently obtained.}.

%We refer to Section \ref{EKres} for relevant definitions)%

\medskip

\begin{Theorem} (Theorem \ref{Theorem_cone})
Suppose $I$ has linear quotients with respect to some ordering $(m_1, \dots, m_k)$ of the generators, and furthermore suppose that $I$ has a regular decomposition function.  Then the minimal resolution of $I$ obtained as an iterated mapping cone is cellular and supported on a regular $CW$-complex.
\end{Theorem}

\medskip

In Section \ref{MoreRes} we investigate other combinatorial types of cellular resolutions that can be recovered from the mapping cone construction.  In particular, we are interested in a family of cellular resolutions in the literature (\cite{CN}, \cite{NR}, \cite{Sine}, \cite{DE}) that all have similar combinatorial structure.  These complexes go by various names and are applied to different classes of ideals, including the \emph{complexes of boxes} resolutions of strongly stable and squarefree strongly stable ideals from \cite{NR}, as well as the \emph{homomorphism complex} resolutions of cointerval hypergraph edge ideals of \cite{DE}.  For a fixed ideal $I$ the combinatorial types of these complexes more or less coincide, but they look much different than the polyhedral complex obtained as the realization of the EK resolution described above (see examples below).

\medskip

We seek to relate these constructions.  In Section \ref{MoreRes} we show that the homomorphism complex resolution can be obtained as an iterated mapping cone with a different choice of `decomposition function'; in essence a different way to glue in the simplex at each step in the construction.  
%In particular we use the fact that the function $g:M(I) \rightarrow G(I)$ need only be defined on monomials in $I$ of the form $x_j u$, where $u \in G(I)$ and $x_j \in \set(u)$.   
In this way we provide a uniform description of various cellular cellular resolutions from the literature, answering a question of Mermin from \cite{Mermin}.  Our main result from that section is the following (see Section \ref{MoreRes} for definitions).

\medskip

\begin{Theorem} (Theorem \ref{Theorem_Homcone})
Suppose $I = I_H$ is a cointerval ideal associated to a cointerval hypergraph $H$, and let $X_H$ be the homomorphism complex supporting its minimal resolution.  Then under the lexicographic ordering of the generators of $I$, the iterated mapping cone resolution of $I$ is supported on $X_H$.
\end{Theorem}

\medskip

The choices involved in realizing the mapping cone as a cellular resolution lead to a natural question of what \emph{all} possible realizations look like.  In Section \ref{RegRes} we fix a choice of basis for each free module in the mapping cone construction and consider the family of geometric realizations obtained by choosing different regular decomposition functions.  The gluing of the simplex at each step amounts to the choice of a map of chain complexes that lifts the given map of $R$-modules.   Finally, in Section \ref{ResSpace} we investigate the extent to which these choices can be realized as a single \emph{space of resolutions}.  %This is reminiscent of the space of cellular resolutions of powers of maximal ideals obtained in \cite{DJS}.

\medskip

\noindent
{\bf Acknowledgments.} The authors would like to thank Alexander Engstr\"om, J\"urgen Herzog, and Volkmar Welker for helpful conversations.

\section{Preliminaries}

For some fixed field $k$, we let $R = k[x_1, \dots, x_n]$ denote the polynomial ring on $n$ variables with its usual ${
\mathbb Z}^n$-grading, and let $I \subset R$ be a monomial ideal. We will be interested in describing (minimal) graded free resolutions of the ideal $I$.

\subsection{Mapping cones and linear quotients}\label{Mapcones}

We begin by recalling the mapping cone resolutions from \cite{HT}.  For this, let $I \subset R$ be a monomial ideal with an ordered set of generators $G(I) = (m_1, \dots, m_k)$ and define $I_j = (m_1, \dots, m_j)$.  Then for each $j$ there are exact sequences 

\[
0 \rightarrow R/(I_{j-1}:m_j) \rightarrow R/I_{j-1} \rightarrow R/I_j \rightarrow 0.
\]

Suppose we have we have resolutions $G$ and $F$ of, respectively, the $R$-modules in the first two positions of the sequence.  We then obtain a resolution of $R/I_j$ as a mapping cone of a homomorphism of complexes $\psi: G \rightarrow F$, where $\psi$ is a lift of the map of $R$-modules $R/(I_{j-1}:m_j) \rightarrow R/I_{j-1}$.  

\medskip

Recall that if $\psi: G \rightarrow F$ is a map of chain complexes then $C(\psi)$, the \emph{mapping cone} of $\psi$, is the chain complex defined by
\[C(\psi) = G[1] \oplus F ,\]
with differential 
\[d_{C(\psi)}(g^{n+1},f^n) = \big( -d_G(g^{n+1}), \psi(g^{n+1}) + d_F(f^n) \big).\]

\medskip

In general the resolution of $R/I_j$ obtained from a mapping cone is not minimal, and so we search for a class of ideals where one can inductively describe both the resolution of $R/(I_{j-1}:m_j)$ as well as the comparison map $\psi$.  In \cite{HT} the authors restrict to a class of ideals for which each of the colon ideals $I_{j-1} : m_j$ are generated by subsets of the variables, in which case the resolution of $R/(I_{j-1} : m_j)$ is resolved by a Koszul complex.  This motivates the following.

\begin{Definition}
A monomial ideal $I \subset R$ is said to have \emph{linear quotients} if there exists an ordering of the generators $I = (m_1, \dots, m_k)$ such that for each $j$ the colon ideal $I_j : m_j$ is generated by a subset of the variables, so that
\[I_j  : m_j = \langle x_{j_1}, \dots, x_{j_r}\rangle. \] 
\end{Definition}

\medskip

One can check (see \cite{HT}, attributed to Sk\"{o}lberg) that an ideal $I$ has linear quotients if and only if the first syzygy module of $I$ has a quadratic Gr\"{o}bner basis.  In the case that $I$ is squarefree, and hence $I$ is the Stanley-Reisner ideal $I_\Delta$ of a simplicial complex $\Delta$, this is equivalent to the Alexander dual $\Delta^*$ being (nonpure) shellable. 

\medskip

%Ideals with linear quotients are related to componentwise linear ideals.  By definition, an ideal $I$ is said to be \emph{componentwise linear} if $I_{\langle j \rangle}$ has a $j$-linear resolution for all $j$, where $I_{\langle j \rangle}$ denotes the ideal generated by all homogeneous polynomials of degree $j$.  In \cite{SV} it is shown that if $I$ is a homogeneous ideal with linear quotients with respect to a \emph{minimal} system of generators $(m_1, m_2, \dots, m_k)$, then $I$ is componentwise linear.
%\medskip

The property of having linear quotients is itself a generalization of the notion of a \emph{stable} ideal, as introduced in \cite{EK}.  If $w$ is a monomial in the polynomial ring $R = k[x_1, \dots, x_n]$ we let $\max(w)$ denote the largest index of the variables dividing $w$.  A monomial ideal $I \subset R$ is called \emph{stable} if for every monomial $w \in I$ and index $i < m = \max(w)$, we have that the monomial $x_i w/x_m$ belongs to $I$.  Stable ideals are in turn generalizations of \emph{strongly stable} (or sometimes \emph{shifted}, or \emph{0-Borel fixed}) ideals.  An ideal $I$ is \emph{strongly stable} if whenever $w \in I$ with $x_j$ dividing $I$, we have $x_i w/x_j \in I$ for all $i <j$.  Note that $j$ is \emph{not} required to be the largest index of variables dividing $w$, as in the definition of stable ideals.

\medskip

\noindent
To summarize, we have the following hierarchy of ideal containments:

$\{ \textrm{strongly stable}\} \subset \{ \textrm{stable} \} \subset \{ \textrm{having linear quotients} \}$. 

%(with respect to a minimal system of generators)} \}$

%$\subset \{ \textrm{componentwise linear} \}$.

\medskip

In \cite{HT} the authors seek \emph{minimal} resolutions of ideals obtained from the mapping cone construction; this amounts to finding an explicit description of the comparison maps $\psi:G \rightarrow F$ introduced above.   The approach taken in \cite{HT} is to first restrict to ideals with linear quotients since in this case one can provide a description of the bases of the free modules in a free resolution.  We first establish some terminology.
\medskip

\begin{Definition}

Suppose $I$ has linear quotients with respect to the sequence $(m_1, \dots, m_k)$ of its generators, and for each $j$ with $ 1 \leq j \leq k$ let $I_j = \langle m_1, \dots, m_j \rangle$.  Let $M(I)$ denote the set of all monomials in $I$.  For each generator $m_j$, with $j = 1, \dots, k$, we define 

\[
\set(m_j) = \{k \in [n]: x_k \in \langle m_1,\dots,m_{j-1}\rangle:m_j\}. 
\]

\end{Definition}

In the case that $I$ has linear quotients, it is shown in \cite{HT} that the mapping cone construction produces a minimal free resolution of $I$, and furthermore that a \emph{basis} for each free module $F_i$ in the minimal free resolution can be explicitly described as follows.

\begin{Lemma}\cite[Lemma 1.5]{HT}
\label{sets}
Let $I$ be a monomial ideal with linear quotients.  Then the iterated mapping cone $F$, derived from the sequence $m_1 ,\ldots, m_k$, is a minimal graded free resolution
of $R/I$, and for all $i > 0$, the symbols
\[
(m;\alpha),\quad\quad {\rm where} \quad m \in G(I) \quad \alpha\subset \set(m),\quad |\alpha|=i-1,,
\]
form a homogeneous basis of the $i^{\rm th}$ module in the minimal resolution of $R/I$.
\end{Lemma}

In \cite{HT} the authors also provide an explicit description of the \emph{differentials} in these resolutions for a certain subclass of ideals that satisfy an extra condition.  In this context, let $M(I)$ denote the set of all monomials in $I$, and define the \emph{decomposition function of $I$} to be the assignment $b: M(I)\rightarrow G(I)$ given by: define $b(m) = m_j$ if $j$ is the smallest number such that $m\in \langle m_1,\ldots,m_{j}\rangle$.  The decomposition function is similar to the description of the differentials in the EK resolution of stable ideals introduced in \cite{EK}.  Indeed, when restricted to the class of stable ideals, the mapping cone construction recovers the EK resolution.  To describe the differentials from \cite{HT} we will need  to assume one further condition.

\begin{Definition} \label{def:regular}
The decomposition function of $I$ is said to be {\emph{regular} }if for each $m\in G(I)$ and every
$t\in \set(m)$ we have
\[
 \quad\set(b(x_t m))\subseteq \set(m).
\]
In this case for each $m\in G(I)$ and $t,s\in \set(m)$
 we have
\[
(*)\quad \quad b(x_s b(x_t m))=b(x_t b(x_s m)).
\]
Note that in degree two for $t,s\in\set(m)$, in the case that $b(x_t b(x_s m))\neq b(x_t m)$, we have
$b(x_t b(x_s m))= b(x_sb(x_t m))=x_t x_s$.
\end{Definition}

\begin{Example}\label{linear_quotients}
Consider the ideal $I = \langle x_1x_3x_4, x_1x_3x_5, x_1x_2x_4, x_1x_4x_5, x_2x_3x_4, x_2x_3x_5 \rangle$.  Then $I$ has linear quotients with respect to this order of generators, where for instance $I_5: (x_2x_3x_5) = \langle x_1, x_4 \rangle$.  In addition, one can check that the decomposition function $b$ for this ideal is regular, where for example $b(x_2x_3(x_1x_4x_5))=b(x_3x_2(x_1x_4x_5))$.  One can also check that $I$ is not stable and also not cointerval (see Section \ref{HomRes} for a definition of the latter). 

\end{Example}

For the class of ideals with linear quotients and regular decomposition functions we have the following result.

\begin{Theorem} \cite[Theorem 1.12.]{HT}
\label{H}
Let $I$ be a monomial ideal with linear quotients, and $F$ the graded
minimal free resolution of $R/I$. Suppose the decomposition function $b : M(I)\rightarrow
G(I)$ is regular. Then the chain map $d$ of $F$ is given by
 \[
     d(m;\alpha) =- \sum_{{j_i}\in \alpha}(-1)^{i-1} x_{j_i} (m;\alpha\setminus {j_i}) +
\sum_{{j_i}\in \alpha}(-1)^{i-1} \frac{x_{j_i} m}{b(x_{j_i}m)} (b(x_{j_i} m);\alpha\setminus {j_i})
\]
\[
=
\sum_{{j_i}\in \alpha}(-1)^{i} x_{j_i} (m;\alpha\setminus {j_i})-
\sum_{{j_i}\in \alpha}(-1)^{i} \frac{x_{j_i} m}{b(x_{j_i}m)} (b(x_{j_i} m);\alpha\setminus {j_i})
\]
if $\sigma\neq \emptyset$, where $\alpha=\{{j_1},\ldots,{j_p}\}\subset \set(m)$ with $j_1<\cdots<j_p$, and
$d(m,\emptyset) = m$ otherwise.
 \end{Theorem}

There are several classes of ideals that are known to have linear quotients that admit regular decomposition functions.  These include stable ideals and matroidal ideals (as shown in (\cite{HT}), as well as \emph{completely lexsegment ideals} \cite{EOS} and the Alexander dual of the generalized Hibi ideals from \cite{EHM}.

\subsection{Cellular resolutions}

One natural way to describe a resolution of an ideal $I$ is through the construction of a polyhedral (or more general CW-) complex ${\mathcal X}_I$ whose faces (vertices, edges, and higher dimensional cells) are labeled by monomials in such a way that the chain complex determining the cellular homology of ${\mathcal X}$ realizes a graded free resolution of $I$.   The study of cellular resolutions was first explicitly initiated in \cite{BayerSturmfels} (to where we refer for further details).  Cellular resolutions have the advantage that algebraic resolutions can in some sense be given a global description, and they also lead to combinatorially interesting geometric complexes.

\medskip

The well-known Taylor resolution \cite{Taylor} guarantees that all monomial ideals have a (usually far from minimal) cellular resolution supported on a simplex.  In this case that $I$ is generated by variables, the Taylor resolution is in fact minimal, and recovers the Koszul resolution of $I$.  This fact will be used heavily in our construction of the resolution of ideals with linear quotients, where by definition certain colon ideals are generated by variables.

\medskip

Not all minimal free resolutions of monomial ideals are supported by a CW-complex, as is shown in \cite{Vel}, but a natural question to ask is which algebraic complexes can be realized cellularly (and to provide a geometric/combinatorial description).   In \cite{BatziesWelker} Batzies and Welker develop an algebraic version of Discrete Morse theory and construct CW complexes that support minimal cellular resolutions of shellable ideals (which can be seen to coincide with the class of ideals with linear quotients).  However, there construction is not explicit and they make no claims regarding the regularity of their complexes.

As mentioned in the introduction, Mermin \cite{Mermin} has shown that the Eliahou-Kervaire resolution of a stable ideal $I$ is cellular and supported on a regular CW-complex.  In the case that $I$ is generated in a fixed degree one can realize the supporting complex as a subcomplex of a certain subdivision of a dilated simplex (see the next section for details).    In \cite{Sine} Sinefakopoulos describes cellular resolutions of the class of strongly stable ideals generated in a fixed degree but obtains combinatorially distinct complexes.  In the case that $I = \langle x_1, \dots x_n \rangle^d$, the $d$th power of the maximal graded ideal in $R$, the complexes are different subdivisions of a dilated simplex.  In \cite{DE} the authors describe cellular resolutions of what they call \emph{cointerval} ideals via spaces of graph homomorphisms, extending the `complexes of boxes' constructions from \cite{NR} and \cite{CN}.  In \cite{NPS} the authors construct minimal cellular resolutions of monomial ideals arising from matroids and oriented matroids. In \cite{DJS} methods from tropical convexity are employed to build cellular resolutions of a certain class of monomial ideals.  For a particular ideal, these constructions can lead to (geometric) complexes with different combinatorial structure.

\medskip

As we will see in Section \ref{MoreRes}, many of these constructions can be realized as different instances of the cellular mapping cone construction.

\section{The generalized Eliahou Kervaire cellular resolution} \label{EKres}

In this section we show that the generalized Eliahou-Kervaire (EK) resolution described in Theorem \ref{H} is supported on a cellular complex.  We will extend the construction of Mermin \cite{Mermin} (where the case of stable ideals was considered) to realize the supporting space as a $CW$-complex obtained by gluing together certain simplices corresponding to data coming from the generators of $I$.   

\medskip

Throughout this section we will use the notation and terminology from Section \ref{Mapcones}.  In particular we assume that our ideal $I$ is a monomial ideal with linear quotients with respect to the sequence of generators $G(I) = (m_1,\ldots,m_k)$, where we use $m_1 > m_2 > \cdots > m_k$ to denote the ordering of the generators defining the linear quotient (this will be convenient when we no longer have indices on the generators).  We furthermore assume that its decomposition function $b: M(I)\rightarrow G(I)$ is regular (see \ref{def:regular}).

\subsection{Cellular construction}
In a description of any cellular resolution we need to label the 0-cells of a CW-complex with monomials corresponding to generators of $I$.  We first note that each monomial $m = x_1^{a_1} x_2^{a_2} \cdots x_n^{a_n} \in G(I)$ can be regarded as a point $(a_1,a_2, \dots a_n) \in {\mathbb R}^n$ by considering its exponent vector.  We will often abuse notation and use $m$ to denote this point in ${\mathbb R}^n$.

Now, for each $m \in G(I)$ and $\alpha=\{{j_1},\ldots,{j_p}\}\subset \set(m)$ with $j_1<\cdots<j_p$, we let $\sigma= (\sigma_1,\ldots,\sigma_p)$ be
a permutation of $\{j_1,\ldots,j_p\}$.  We define ${\rm ch}(m,\alpha,\sigma)$ to be the subset of ${\mathbb R}^n$ obtained as the convex hull of the subset of generators of $I$ that we reach by applying the decomposition function in the order prescribed by the permutation $\sigma$:
\[
{\rm ch}(m,\alpha,\sigma) = {\rm Conv} \big(m,  b(x_{\sigma_1}m),  b(x_{\sigma_2}x_{\sigma_1}m),\ldots, b(x_{\sigma_p}\cdots x_{\sigma_1}m) \big)
\]
Here we use the shorthand notation $b(x_{\sigma_2}x_{\sigma_1}m) = b(x_{\sigma_2}b(x_{\sigma_1}m))$.  We say that ${\rm ch}(m,\alpha,\sigma)$ is \emph{nondegenerate} if there are no repetitions of monomials involved in the description of ${\rm ch}(m,\alpha,\sigma)$; otherwise we say that ${\rm ch}(m,\alpha,\sigma)$ is \emph{degenerate} (see Example \ref{exam:generate}).  We will see in Corollary \ref{cor:simplex} that in this case ${\rm ch}(m,\alpha,\sigma)$ is in fact a $p$-dimensional simplex.

\medskip

\begin{Lemma} \label{degenerate}
Let $m \in G(I)$, $\alpha \subset \set(m)$, and $\sigma$ a permutation of $\alpha$. If $\ch(m,\alpha,\sigma)$ is degenerate then there exists a permutation
$\sigma^{\prime}$ of $\alpha$ such that
$\ch(m,\alpha,\sigma^{\prime})$ is nondegenerate and $\ch(m,\alpha,\sigma)$ is a face of $\ch(m,\alpha,\sigma^{\prime})$.

\end{Lemma}

\begin{proof}
Suppose that $\ch(m,\alpha,\sigma)$ is degenerate.
Let $s$ be the first index with the property that applying $b$ does not lead to a new generator, so that 
\[
b(x_{\sigma_{s+1}} b(x_{\sigma_s}\cdots x_{\sigma_1}m)) = b(x_{\sigma_s}\cdots x_{\sigma_1}m).
\]

%First assume that $s$ be the first index such that for  $u=b(x_{\sigma_s}\cdots x_{\sigma_1}m)$ there exists an index $t\geq s$ with
%\[
%b(x_{\sigma_{t+1}}b(x_{\sigma_t}\cdots x_{\sigma_{s+1}}u))=u.
%\]
%Assume that $t$ is the smallest index with this property. We claim that $t=s$. By contrary assume that $t>s$.
%Therefore
%\[
%u=b(x_{\sigma_{t+1}}b(x_{\sigma_t}\cdots x_{\sigma_{s+1}}u))\geq \cdots  \geq b(x_{\sigma_{s+2}} b(x_{\sigma_{s+1}}u))\geq b(x_{\sigma_{s+1}} u)> u,
%\]
%which is a contradiction. So we  have $t=s$ and
%\[
%b(x_{\sigma_{s+1}} b(x_{\sigma_s}\cdots x_{\sigma_1}m)) = b(x_{\sigma_s}\cdots x_{\sigma_1}m).
%\]
Let $\sigma^\prime=(\sigma_1,\ldots,\sigma_{s-1},\sigma_{s+1},\sigma_s,\ldots,\sigma_p)$ be the permutation
obtained from $\sigma$ from switching positions $s$ and $s+1$.   

If $b(x_{\sigma_{s+1}}b(x_{\sigma_{s-1}}\cdots x_{\sigma_{1}}m))\neq b(x_{\sigma_{s-1}}\cdots x_{\sigma_{1}}m)$ then we get that $\ch(m,\alpha,\sigma)$ is a face of $\ch(m,\alpha,\sigma^\prime)$.  Otherwise we push $\sigma_{s+1}$ further back by switching $\sigma_{s-1}$ and $\sigma_{s+1}$ to get the permutation $\sigma^{\prime \prime} = (\sigma_1, \dots, \sigma_{s+1},\sigma_{s-1}, \sigma_s, \dots, \sigma_p)$.  We then apply the above argument for $s-1$ (instead of $s$) and in order to find a permutation $\sigma^{\prime \prime}$ such that $\ch(m,\alpha,\sigma'')$ contains $\ch(m,\alpha,\sigma')$ as a face.  In the worst case we consider the permutation $(\sigma_{s+1}, \sigma_1, \dots, \sigma_p)$ and we have $b(x_{\sigma_{s+1}}m) \neq m$ as desired.

By continuing the same argument, we keep increasing dimensions and we get a nondegenerate $\ch(m,\alpha,\sigma')$ containing $\ch(m,\alpha,\sigma)$ as a face.
\end{proof}

\medskip

\begin{Example}\label{exam:generate}
We return to the ideal $I$ from Example \ref{linear_quotients}, where 
\[I = \langle x_1x_3x_4, x_1x_3x_5, x_1x_2x_4, x_1x_4x_5, x_2x_3x_4, x_2x_3x_5 \rangle\]
\noindent
has linear quotients with respect to the given ordering of the generators.  One can check that $\set(x_1x_4x_5) = \{2,3\}$.  The simplices corresponding to the permutations $\sigma = (23)$ and $\sigma' = (32)$ are described below.
\begin{itemize}
\item [(i)] The simplex ${\rm ch}(x_1x_4x_5,\{2,3\},(32))$ is  {\it degenerate} since it is the convex hull of the vectors corresponding to $x_1x_4x_5$ and $x_1x_3x_4=b(x_3(x_1x_4x_5))$.  Here the fact that $2\not\in \set(x_1x_3x_4)$ implies degeneracy (see Figure \ref{fig_degen}).

\smallskip

\item[(ii)] 
Switching the positions of 2 and 3 gives us the permutation $\sigma' = (32)$.  The resulting simplex ${\rm ch}(x_1x_4x_5,\{2,3\},(32))$ is then {\it nondegenerate} since it is the convex hull of the vectors corresponding to $x_1x_4x_5, x_1x_2x_4=b(x_2(x_1x_4x_5))$, and $x_1x_3x_4=b(x_3(x_1x_2x_4))$.  Note that ${\rm ch}(x_1x_4x_5,\{2,3\},\sigma)$ is a face of ${\rm ch}(x_1x_4x_5,\{2,3\},\sigma')$.

\end{itemize}

\end{Example}

%\textcolor{blue}{I switched around the examples a bit, combined two into one and also removed the other example with the degree 3 generators (it's still below but commented out).  Do you think we need it?}

\begin{figure}\label{fig_degen}
\begin{center}
  \includegraphics[scale = 0.9]{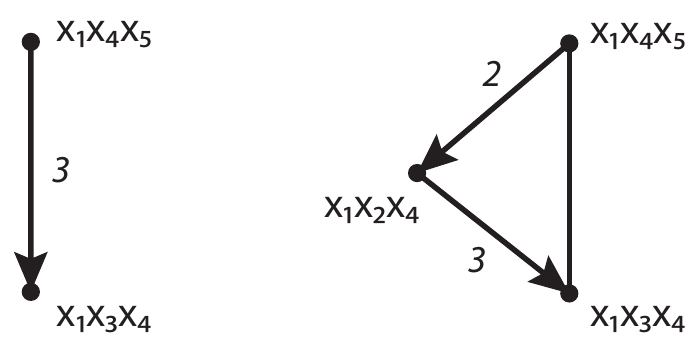}
\label{Degenerate}
    \caption{The degenerate simplex ${\rm ch}(x_1x_4x_5,\{2,3\},(32))$ on the left, and the nondegenerate ${\rm ch}(x_1x_4x_5,\{2,3\},(23))$ on the right.}

\end{center}
\end{figure}

%\begin{Example}\label{exam:generate2}
%Consider the ideal $I = \langle x_1x_2x_3x_5, x_2x_3x_4x_5, x_2x_3x_5x_6, x_2x_3x_5x_7 \rangle$ in this ordering of the generators. Then $\sigma=(641)$ is the only permutation of $\{1,4,6\}$ such that ${\rm ch}(x_2x_3x_5x_7,\{1,4,6\},\sigma)$ is {\it nondegenerate}, since $6\in\set(m)$ only if $m=x_2x_3x_5x_7$. In particular ${\rm ch}(x_2x_3x_5x_7,\{1,4,6\},(416))$ is {\it degenerate}.  Let $\sigma'=(461)$, obtained from $\sigma$ by switching the positions of the elements $6$ and $1$ since we lose the nondegenerate property in $6$.  The resulting ${\rm ch}(x_2x_3x_5x_7,\{1,4,6\},\sigma')$ is still degenerate, so we switch the elements $6$ and $4$ to obtain $\sigma''=(641)$.  In the end we see that ${\rm ch}(x_2x_3x_5x_7,\{1,4,6\},\sigma'')$ is nondegenerate containing the original cell as a face.
%\end{Example}

\medskip

\begin{Lemma}
\label{primary}
Suppose $\ch(m,\alpha,\sigma)$ is nondegenerate.
Then for each $\ell$ and $j$ with $j< \ell$ we have
\[
\quad b(x_{\sigma_{j}} x_{\sigma_{j-1}}\cdots x_{\sigma_{1}} m)\neq x_{\sigma_{j}} b(x_{\sigma_{j-1}}\cdots x_{\sigma_{1}} m)/x_{\sigma_{\ell}}.
\]
\end{Lemma}

\begin{proof}
By contradiction assume that for some $j<\ell$ we have
\[
\quad b(x_{\sigma_{j}} m')=x_{\sigma_{j}} m'/x_{\sigma_{\ell}},
\]
where $m'=b(x_{\sigma_{j-1}}\cdots x_{\sigma_{1}} m)$.
Then by our assumption of regularity $(*)$ we can push $\sigma_{\ell}$ back until it meets $\sigma_j$ as follows:
\begin{eqnarray*}
b(x_{\sigma_{\ell}}\cdots x_{\sigma_{1}}m)
&=&
b(x_{\sigma_{\ell}} b(x_{\sigma_{\ell-1}}\cdots x_{\sigma_{j+1}} b(x_{\sigma_{j}}m')))\\
&=&
b(x_{\sigma_{\ell-1}} b(x_{\sigma_{\ell}}x_{\sigma_{\ell-2}}\cdots x_{\sigma_{j+1}} b(x_{\sigma_{j}}m')))
\\
 &=&
\cdots
\\
&=&
b(x_{\sigma_{\ell-1}} b(x_{\sigma_{\ell-2}}\cdots x_{\sigma_{j+1}} x_{\sigma_{\ell}} b(x_{\sigma_{j}}m')))
\\
&=&
b(x_{\sigma_{\ell-1}}\cdots x_{\sigma_{j+1}} x_{\sigma_{\ell}} (x_{\sigma_{j}}m'/x_{\sigma_{\ell}}))
\\
&=&
b(x_{\sigma_{\ell-1}}\cdots x_{\sigma_{1}}m),
\end{eqnarray*}
which is a contradiction by our assumption that $\ch(m,\alpha,\sigma)$ is nondegenerate.
\end{proof}

As a consequence of Lemma~\ref{primary} we have

\medskip

\begin{Corollary} \label{cor:simplex}
Suppose that $|\set(m)| = p$ and $\ch(m,\set(m),\sigma)$ is nondegenerate. Then $\ch(m,\set(m),\sigma)$ is a simplex of dimension $p$.
\end{Corollary}

\begin{proof}
\label{lem}
Let
$v_0 = (v_{0,1}, v_{0,2},\ldots,v_{0,n})$ be the exponent vector of $m$ and $v_i = (v_{i,1}, v_{i,2},\ldots, v_{i,n})$
be the exponent vector of $b(x_{\sigma_i}\cdots x_{\sigma_1}m)$ for each $i>0$.
We will show that $v_0,v_1,\ldots,v_p$ are affinely independent. Assume that $\lambda_0 v_0+\lambda_1 v_1+\cdots+\lambda_p v_p=0$ and $\lambda_0+\lambda_1+\cdots+\lambda_p=0$. Then by Lemma~\ref{primary} 
 we know that for each $\ell$,
\[
v_{i,\sigma_\ell}=
\begin{cases}
v_{\ell-1,\sigma_\ell}, &\text{if $i< \ell$;}\\
v_{\ell-1,\sigma_\ell}+1, &\text{if $i\geq \ell$.}
\end{cases}
\]

In particular $v_{p,\sigma_p}=v_{\ell,\sigma_p}+1$ for all $\ell<p$.
Therefore, we have 
\[
(v_{p,\sigma_p}-1)(\lambda_0 +\lambda_1+\cdots+\lambda_{p-1})+v_{p,\sigma_p}\lambda_p=0
\]
which implies that $\lambda_p=0$. A similar argument shows that $\lambda_i=0$ for all $i$.
\end{proof}

As we have seen, for each subset $\alpha \subset \set(m)$ of size $q$ we have some permutation $\sigma$ of $\alpha$ such that the simplex $\ch(m,\alpha,\sigma)$ is of dimension $q-1$.

%\begin{Corollary}
%\label{lem}
%Suppose that $|\set(m)| = p$ and $\ch(m,\set(m),\sigma)$ is nondegenerate.
%Let
%$v_0 = (v_{0,1}, v_{0,2},\ldots,v_{0,n})$ be the exponent vector of $m$, $v_i = (v_{i,1}, v_{i,2},\ldots, v_{i,n})$
%be the exponent vector of $b(x_{\sigma_i}\cdots x_{\sigma_1}m)$.
% and $k=\max(b(x_{\sigma_{p}}\cdots x_{\sigma_{1}}m))$.
%Then we have

%\[
%v_{\ell,\sigma_\ell}=v_{j,\sigma_\ell}=v_{\ell-1,\sigma_\ell}+1
%\quad {\rm for\ all\ } j\geq \ell\ {\rm and }\
%v_{i,\sigma_\ell}=v_{\ell-1,\sigma_\ell}\ {\rm for\ all\ } i< \ell.
%\]
%\\
%In particular $v_{p,\sigma_\ell}=v_{\ell,\sigma_\ell}$ for all $\ell$.
%\end{Corollary}

%\begin{Corollary}\label{description}
%Let $z = (z_1,\ldots,z_n)\in \ch(m,\set(m),\sigma)$ with $z=\sum_{i=0}^p c_i v_i$, $\sum_{i=1}^p c_i=1$ and $c_i>0$ for all
%$i$.
%Then we can find the coefficients $c_j$ by solving the set of the following equations for all $\ell$:
%\[
%z_{\sigma_\ell}=(c_p+\cdots+c_\ell) v_{p,\sigma_\ell}+\sum_{j=1}^{\ell-1} c_j (v_{p,\sigma_\ell}-1)
%\]

%\begin{eqnarray*}
%z_{\sigma_p}&=&c_p v_{p,\sigma_p}+\sum_{j=1}^{p-1} c_j (v_{p,\sigma_p}-1)\\
%z_{\sigma_{p-1}}&=&(c_p+c_{p-1}) v_{p,\sigma_{p-1}}+\sum_{j=1}^{p-2}  c_j (v_{p,\sigma_{p-1}}-1)\\
%&&\cdots\\
%z_{\sigma_\ell}&=&(c_p+\cdots+c_\ell) v_{p,\sigma_\ell}+\sum_{j=1}^{\ell-1} c_j (v_{p,\sigma_\ell}-1)\\
%\end{eqnarray*}

%\end{Corollary}

\begin{Definition}
Suppose $m$ is a generator of $I$ and $\alpha \subset \set(m)$.  Let $F$ be a facet of the (possibly degenerate) simplex $\ch(m,\alpha,\sigma)$.  We say that $F$ is an \emph{exterior facet} if it is not a facet of $\ch(m,\alpha,\sigma')$ for some other permutation $\sigma'$ of $\alpha$.  Otherwise we say that $F$ is an \emph{interior facet}.
\end{Definition}

%\begin{Definition}
%We say that a facet $F$ of $\ch(m,\set(m),\sigma)$ is an \emph{exterior facet} if it is not a facet of $\ch(m,\set(m),\sigma')$ for another permutation $\sigma'$ of $\set(m)$.  Otherwise we say that $F$ is an \emph{interior facet}.
%\end{Definition}

\begin{Corollary}
\label{inter}
Let $\alpha \in \set(m)$, let $\sigma$ be some permutation of $\alpha$, and suppose $F$ is a facet of the simplex $\ch(m, \alpha,\sigma) = {\rm Conv}\{m,m_1,\ldots,m_p\}$.  Then $F$ is exterior if and only if $m_\ell\not\in F$ for some $1<\ell<p$ and $\sigma_{\ell}\in\set(b(x_{\sigma_{\ell+1}} m_{\ell-1}))$.
If $F$ is interior, then there exist exactly two nondegenerate simplices containing $F$: the simplex $\ch(m,\alpha,\sigma)$ and the simplex $\ch(m,\alpha,\sigma')$, where $\sigma'=(\sigma_1,\ldots,\sigma_{\ell-1},\sigma_{\ell+1},\sigma_\ell,\ldots,\sigma_p)$.

%Let $F$ be a facet of 
%\[
%\bigcup_{\ch(m,\alpha,\sigma)\atop{{\rm nondegenerate}}}\ch(m,\alpha,\sigma),
%\]
%where $\alpha\subseteq \set(m)$.
%Then $F$ is interior if and only if there exists $\ch(m,\alpha,\sigma)=\langle m,m_1,\ldots,m_p\rangle$ containing $F$ as a facet with $m_\ell\not\in F$ for some $1<\ell<p$
%and
%$\sigma_{\ell}\in\set(b(x_{\sigma_{\ell+1}} m_{\ell-1}))$.
%If $F$ is interior, then there exist exactly two nondegenerate cells containing $F$; the cells $\ch(m,\alpha,\sigma)$ and $\ch(m,\alpha,\sigma')$, where $\sigma'=(\sigma_1,\ldots,\sigma_{\ell-1},\sigma_{\ell+1},\sigma_\ell,\ldots,\sigma_p)$.

\end{Corollary}

%\textcolor{blue}{Again, is this right?  I reworded the Lemma a bit but I think this is what we're trying to say?  Also note that throughout the statement and proof I changed 'cell' to `simplex' when describing the $\ch(m, \alpha,\sigma)$.  I think cell should be reserved for the actual cells we get after gluing together the simplices.  Does that make sense?}

\begin{proof}
Let $\ch(m,\alpha,\sigma)=\langle m,m_1,\ldots,m_p\rangle$ be a nondegenerate simplex which contains $F$ as a facet with $m_\ell\not\in F$ for some $1< \ell<p$. Now we consider the following two cases:

\medskip

\textbf{Case (i).} Suppose $\sigma_{\ell}\in\set(b(x_{\sigma_{\ell+1}} m_{\ell-1}))$.  Then $b(x_{\sigma_{\ell}} x_{\sigma_{\ell+1}} m_{\ell-1}) \neq b(x_{\sigma_{\ell+1}} m_{\ell-1})$ and so  $F$ is also a facet of $\ch(m,\set(m),\sigma')$, where $\sigma'=(\sigma_1,\ldots,\sigma_{\ell-1},\sigma_{\ell+1},\sigma_{\ell},\ldots,\sigma_{p})$. Therefore $F$ is an interior facet.

By contradicton assume that
$\ch(m,\set(m),\sigma'')$
is another nondegenerate simplex containing $F$ as a facet such that $\sigma''_{i}=\sigma_{\ell}$.
If $i<\ell$, then the corresponding simplex does not contain $m_{\ell-1}$.
If $i>\ell$, then $m_{\ell+1}$ does not appear among the vertices, a contradiction.

\medskip

\textbf {Case (ii).} Suppose $\sigma_{\ell}\not\in\set(b(x_{\sigma_{\ell+1}} m_{\ell-1}))$.  In this case we have $b(x_{\sigma_{\ell}} x_{\sigma_{\ell+1}} m_{\ell-1}) = b(x_{\sigma_{\ell+1}} m_{\ell-1})$.
Assume that $\ch(m,\set(m),\sigma')=\langle w_0,w_1,\ldots,w_p\rangle$ is another simplex containing $F$ as a facet and let $w_i\not\in\{m,m_1,\ldots,m_p\}$. We wish to show that $i=0$.  By contradiction we suppose $i\geq 1$ and note that $b(x_{\sigma_\ell} w_{i-1})=w_i$. 
Our assumption that $b(x_{\sigma_{\ell}} x_{\sigma_{\ell+1}} m_{\ell-1}) = b(x_{\sigma_{\ell+1}} m_{\ell-1})$ implies that $\sigma_\ell\in\set(m_i)$ for all $i\leq \ell$ and $\sigma_\ell\not\in\set(m_i)$ for all $i>\ell$. On the other hand, $\sigma_\ell\in\set(w_{i-1})$ which implies that $i<\ell$ and $\sigma_{\ell+1}\in\set(w_i)$. Thus  $w_{i+1}=b(x_{\sigma_{\ell+1}} w_i)$ is not equal to $m_{\ell+1}$ which implies that $w_i,w_{i+1}\not\in F$, a contradiction.

Therefore we have $\ell=0$ so that $F=\langle m_1,\ldots,m_p \rangle$ is the facet of $\ch(m, \alpha,\sigma)$ obtained by removing the vertex $m$.
We conclude that $F$ is not a facet of any other simplex of the form $\ch(m,\alpha,\sigma')=\langle m,m'_1,\ldots,m'_p\rangle$, and hence $F$ is exterior. 
\end{proof}

%\begin{Example}
%Consider the facet on the vertices $45,25,12$ of $\ch(x_4x_5,\set(x_4x_5),(x_2,x_3,x_1))$ depicted in Figure~1.
%$F$ is exterior by Lemma~\ref{inter}(ii) since $j=3\neq {\sigma_1}$, and there exists no $w$ with $b(x_3 w)=x_4x_5$. 
%\end{Example}

%\begin{Example}
%Consider the facet on the vertices $345,145,125$ of $\ch(m,\set(m),\sigma)$ where $m=x_3x_4x_5,%\set(m)=\{x_1,x_2,x_3\}$ and $\sigma=(x_1,x_2,x_3)$ depicted in Figure~.
%$F$ is interior by Lemma~\ref{inter}(ii) since $b(x_3 w)=m$ for $w=x_4x_5x_6$, and $F$ is a facet of
%$\ch(w, \set(w),\sigma')$ for $\sigma'=(x_3,x_1,x_2)$.
%\end{Example}

We next construct the cells that will serve as basis elements of the free modules in our resolution.  We obtain these by gluing together the simplices  ${\rm ch}(m,\alpha,\sigma)$ corresponding to the different choices of the permutation $\sigma$.  For this we define the cell $U(m, \alpha)$ as the union over all permutations $\sigma$ of $\alpha$.

\[ U(m, \alpha) = \bigcup_{\textrm{$\sigma$ a permutation of $\alpha$}} {\rm ch}(m,\alpha,\sigma) .\]

Note that by Lemma \ref{degenerate}, $U(m,\alpha)$ can be written (as a subset of ${\mathbb R}^n$) as the union of \emph{nondegenerate} simplices
$\ch(m,\alpha,\sigma)$.

\subsection{ Orientation of  $\ch(m, \set(m),\sigma)$:}

\medskip

\medskip

\begin{Definition}
For each $p\geq 1$ we fix the permutation $(p,\ldots,1)$. Then for each permutation $\sigma$ of $\{1,\ldots,p\}$ we define
\[\epsilon(\sigma)=\sgn(\sigma,(p,\ldots,1))\] 
where $\sgn(-,-)$ denotes the standard sign function for permutations.
\end{Definition}

\begin{Lemma}\label{orient}
There exists an orientation on the simplices $\ch(m,\set(m),\sigma)$ such if $F$ in an interior facet belonging to both $\ch(m,\set(m),\sigma)$ and $\ch(m,\set(m),\sigma')$, then $F$ has the same induced orientation if and only if the the simplices themselves have opposite orientations.
\end{Lemma}

\begin{proof}
By Lemma~\ref{inter} each interior facet $F$ belongs to exactly two simplices $\ch(m,\set(m),\sigma)$ and 
$\ch(m,\set(m),\sigma')$. Therefore the orientation of each interior facet $F$ can be determined uniquely by the simplices containing $F$.

For each two nondegenerate simploices $G=\ch(m,\set(m),\sigma)$ and $G'=\ch(m,\set(m),\sigma')$ there exists a chain $G_0=G,G_1,\ldots,G_r=G'$ such that $G_{i+1}$ is obtained from $G_i$ by swapping two suitable indices. 

These two facts show that the orientations on the simplices can be determined uniquely.
\end{proof}

\medskip

\begin{Notation}
 Assume that $F$ is a facet of $\ch(m,\set(m),\sigma)=\langle m= m_0, m_1, \ldots,m_p\rangle$. 
% The  orientation of $F$ in this cell is denoted by $\epsilon(F,\sigma)=(-1)^{p-\sigma_i}\sgn(\sigma)$, where 
% $F=\langle m_0,\ldots,{\widehat{m_i}},\ldots,m_p\rangle$.
 When $F$ is an exterior facet, or when the simplex containing $F$ is clear, we denote $o(F)$ to denote the the orientation of $F$ that is determined by $\sgn(\sigma)$ and the missing vertex in $F$.
\end{Notation}

The following is our main technical lemma regarding the differential maps.

\begin{Lemma} \label{Lemma_diff}
The topological differentials of each cell $U(m,\alpha)$ with subset $\alpha=\{{j_1},\ldots,{j_p}\}$ of $\set(m)$ are
\[
d(U(m,\alpha))=\sum_i (-1)^i U(m,\alpha\backslash j_i) - \sum_i (-1)^i U(b(x_{j_i}m),\alpha\backslash j_i)\ .
\]
\end{Lemma}

\begin{proof}
For the simplex $\langle m_0,m_1,\ldots,m_p \rangle$, the boundary map is given by

\[
d(\langle m_0,m_1,\ldots,m_p \rangle)= \sum_{i=0}^p (-1)^{p-i}\langle m_0,\ldots,{\widehat{m_i}},\ldots,m_p \rangle
\]

\medskip

For $\alpha \subset \set(m)$, the cell $U(m,\alpha)$ for $\alpha\subseteq \set(m)$ is the union of all nondegenerate simplices $\ch(m,\alpha,\sigma)$, and is oriented as
\[
U(m,\alpha)=\sum_{F=\ch(m,\alpha,\sigma)\atop{{\rm nondegenerate}}}
o(F) F,
\]
where $o(F)$ denotes the unique orientation given by Lemma~\ref{orient}. So the topological differentials of the cell $\alpha\subseteq\set(m)$ are given by

\[
d(U(m,\alpha))= \sum_{F=\ch(m,\alpha,\sigma)\atop{{\rm nondegenerate}}}
o(F) d(F).
\]

Now corresponding to the fixed simplex
$\ch(m,\alpha,\sigma)=\langle m_0,m_1,\ldots,m_p \rangle$ we have the following terms in the differential of $U(m,\alpha)$:

\medskip

Case $1.$ $i=0$.
The term corresponding to $i=0$ is
$(-1)^p \sgn(\sigma) \langle m_1,\ldots,m_p \rangle$. 
\\
Let $\sigma'=({\sigma_2},\ldots,{\sigma_p})$ and $\pi$ be the permutation such that $\pi\sigma'=(p,\ldots,\widehat{{\sigma_1}},\ldots,1)$.
Then
\[
({\sigma_1},p)({\sigma_1},{p-1})\cdots ({\sigma_1},{\sigma_1+1})\pi \sigma'= (p,\ldots,1)
\]
shows that $\sgn(\sigma)=\sgn(\sigma') (-1)^{p-\sigma_1}$. Thus the corresponding term is
\[
(-1)^{\sigma_1} \sgn(\sigma')\ch(b(x_{j_{\sigma_1}} m),\alpha\backslash j_{\sigma_1},\sigma').
\]
This face can not be obtained from the differential of
any other simplex in $U(m,\alpha)$ by Lemma~\ref{inter}.

\medskip

Case $2$. $i=p$. In this case the term corresponding to $i=p$ is $\sgn(\sigma) \langle m_0,\ldots,m_{p-1}\rangle$.
\\
Then again this face can not be obtained from the differential of
any other simplex in $U(m,\alpha)$.
As in  Case $1$ we have $\sgn(\sigma)=\sgn(\sigma')(-1)^{\sigma_p+1}$, where $\sigma'=(\sigma_1,\ldots,\sigma_{p-1})$. Thus this term can be written as
\[
(-1)^{\sigma_p+1}\sgn(\sigma') \ch(m,\alpha\backslash j_{\sigma_p},\sigma').
\]

\medskip

Case $3$. $0<i<p$. The corresponding term is
\[
(-1)^{p-i}\sgn(\sigma)\langle m_0,\ldots,{\widehat{m_i}},\ldots,m_p \rangle.
\]

Now we have two subcases:

\medskip

Case $3.1$. First suppose $\sigma_i\in\set(b(x_{j_{\sigma_{i+1}}}m_{{i-1}}))$. Thus $b(x_{\sigma_i}b(x_{j_{\sigma_{i+1}}}m_{{i-1}}))\neq b(x_{j_{\sigma_{i+1}}}m_{{i-1}})$.
We set $\sigma'=(\sigma_1,\ldots,\sigma_{i-1},
\sigma_{i+1},\sigma_{i},\ldots,\sigma_p)$. Our condition guarantees that $\ch(m,\alpha,\sigma')$ is nondegenerate.
Note that $\ch(m,\alpha,\sigma')=\langle m_0,\ldots,m_{i-1},m'_{i},m_{i+1},\ldots,m_p\rangle$ for some $m'_{i}$. Now by removing the $i^{\rm th}$ vertex of $\ch(m,\alpha,\sigma')$ we get
\[
(-1)^{p-i}\sgn(\sigma') \langle m_0,\ldots,
{\widehat{m'_i}},\ldots,m_p\rangle.
\]
Since $\sgn(\sigma')=-\sgn(\sigma)$, when we take the sum over all possible permutations of $\alpha$ given nondegenerate cells, these two terms  will be canceled.

\medskip

Case $3.2$. Next suppose $\sigma_i\not\in\set(b(x_{j_{\sigma_{i+1}}}m_{{i-1}}))$. Then $b(x_{\sigma_i}b(x_{j_{\sigma_{i+1}}}m_{{i-1}}))= b(x_{j_{\sigma_{i+1}}}m_{{i-1}})$ and the corresponding term is 
\[
 (-1)^{p-i}\sgn(\sigma)\ch(m,\alpha\backslash j_{\sigma_i},\sigma'),
\]
where $\sigma'=(\sigma_1,\ldots,\widehat{\sigma_i},\ldots,\sigma_p)$. Since $(-1)^{p-i-\sigma_i+1}\sgn(\sigma')=\sgn(\sigma)$, we can write this term as
\[
 (-1)^{\sigma_i+1}\ch(m,\alpha\backslash j_{\sigma_i},\sigma').
\]
%We have $(-1)^{p-i-\sigma_i+1}o(\sigma')=o(\sigma)$, since $\sigma_i$ should be transferred to the position $p-i+1$. If %$\sigma_i=p-i+1$, then $o(\sigma)=o(\sigma')$. Otherwise in some finite steps by swapping $\ell$ with $\ell+1$ we get the result.

\medskip

Now by considering $\sum d(o(F)F)$ over all non-degenerate cells $F$ we find that
the remaining terms are the sum of
\\

\noindent
$
(1) \sum_i (-1)^{i} o(\sigma') \ch(b(x_{j_i} m),\alpha\backslash j_i,\sigma'),
$
\\

\noindent
$
(2) \sum_i (-1)^{i+1} o(\sigma') \ch(m,\alpha\backslash j_i,\sigma'),
\  {\rm where}\
j_i\in\set(b(x_{j_{i+1}}b(x_{j_{i-1}}\cdots x_{j_1}m)))$,
\\

\noindent
$
(3.2) \sum_i (-1)^{i+1} o(\sigma') \ch(m,\alpha\backslash j_i,\sigma'),
\  {\rm where}\
j_i\not\in\set(b(x_{j_{i+1}}b(x_{j_{i-1}}\cdots x_{j_1}m)))$
\\
\medskip

\noindent
over all $\sigma'$ where the corresponding facet is nondegenerate.
%$\sigma'=(\sigma_1,\ldots, \widehat{i},\ldots, \sigma_p)$.
Then the first sum can be written as
$\sum_i (-1)^i U(b(x_{j_i}m),\alpha\backslash j_i)$
and the sums coming from $(2)$ and $(3.2)$ can be written as
$\sum_i (-1)^{i+1} U(m,\alpha\backslash j_i)$. Therefore
\[
d(U(m,\alpha))=\sum_{i} (-1)^i U(b(x_{j_i}m),\alpha\backslash j_i)+\sum_i (-1)^{i+1} U(m,\alpha\backslash j_i),
\]

\noindent
as desired.  This completes the proof.

\end{proof}

With these preliminaries in place we can establish the main result of this section.

\begin{Theorem} \label{Theorem_cone}
Suppose $I$ has linear quotients with respect to some ordering $(m_1, \dots, m_k)$ of the generators, and furthermore suppose that $I$ has a regular decomposition function.  Then the minimal resolution of $I$ obtained as an iterated mapping cone is cellular and supported on a regular $CW$-complex.\end{Theorem}

\begin{proof}

Adding the monomial coefficients to the differential map from Lemma \ref{Lemma_diff} we obtain
\[
d(U(m,\alpha))=\sum_{i} (-1)^i (\frac{x_{j_i}m}{b(x_{j_i}m)})U(b(x_{j_i}m),\alpha\backslash j_i)-\sum_i (-1)^{i} x_{j_i} U(m,\alpha\backslash j_i).
\]
This is precisely the minimal free resolution of $I$ described in Theorem~\ref{H}.  Therefore the complex constructed as the union of the cells $U(m,\alpha)$ supports the minimal free resolution of $I$, as desired.  Moreover this resolution in the closed form looks like the Eliahou-Kervaire resolution.

\medskip

Note that for any two simplices $G=\ch(m,\set(m),\sigma)$ and $G'=\ch(m,\set(m),\sigma')$ there exists the chain $G_0=G,G_1,\ldots,G_r=G'$ where $G_{i+1}$ can be obtained from $G_i$ by swapping two suitable indices. This implies that $U(m,\set(m))$ is a shellable for all $m$.
Hence the constructed resolution is regular, by Lemma~\ref{inter} and \cite[Proposition~1.2]{DK}.
\end{proof}

\begin{figure}
\begin{center}
  \includegraphics[scale = .85]{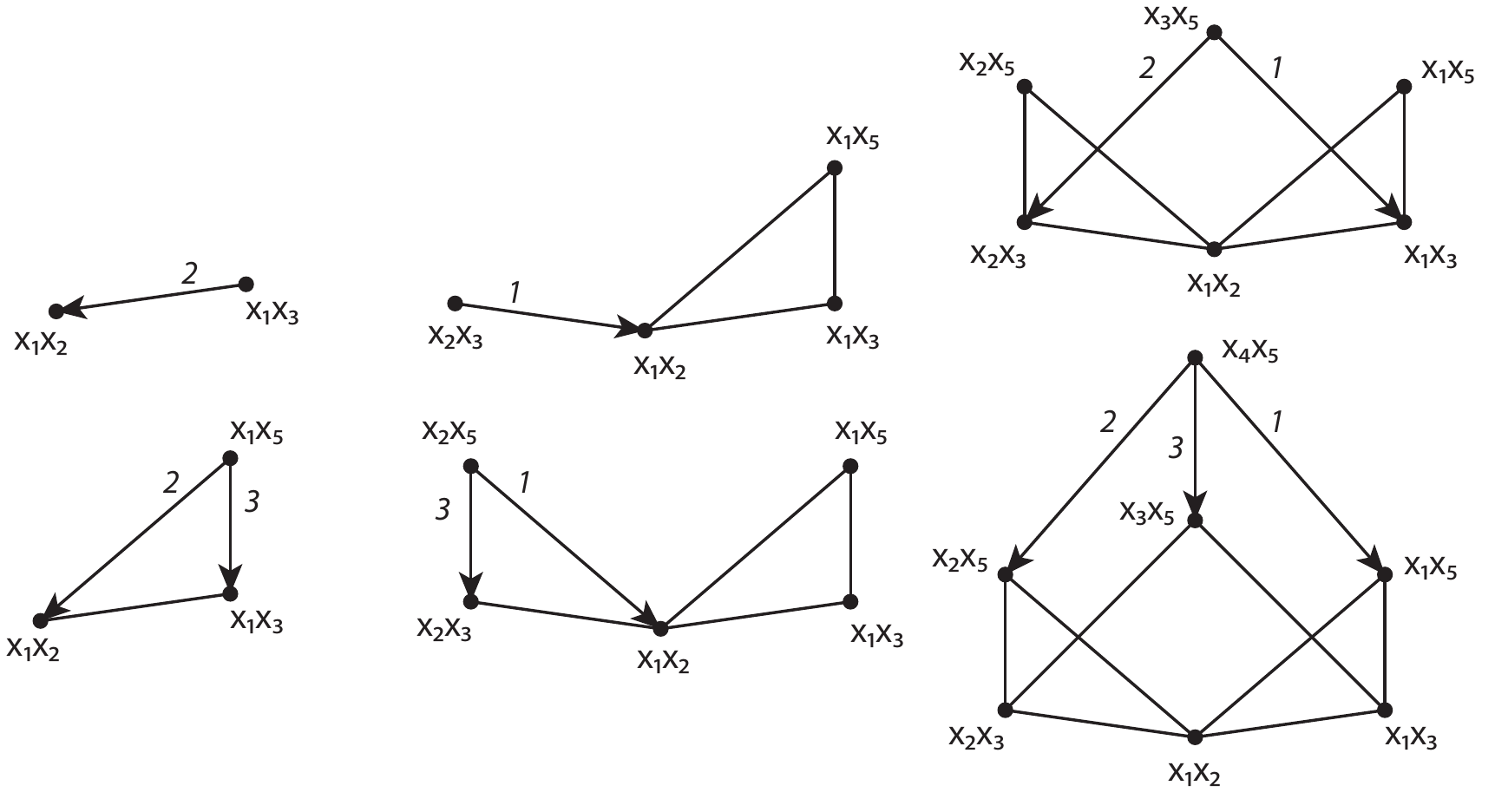}
\label{EKresolution}
    \caption{The resolution of $I = \langle x_1x_2, x_1x_3,x_1x_5, x_2x_3, x_2x_5, x_3x_5, x_4x_5 \rangle$ built from an iterated mapping cone.  At each step we have labeled the newly added edges with the elements of $\set(m_j)$.}

\end{center}
\end{figure}

\begin{Remark}

We note that in the proofs of Theorem \ref{Theorem_cone} and the related lemmas, the only property of the decomposition function that we use is its regularity, and not its definition in terms of assigning a particular generator to a monomial.  Hence we obtain similar results for any decomposition type function that satisfies the regularity property.  We return to this point in Section \ref{RegRes} where we vary the decomposition function to obtain combinatorially distinct cellular resolutions.

\end{Remark}

\section{Other cellular realizations of the mapping cone}\label{MoreRes}

In each step of the mapping cone construction we need to choose the homomorphism of complexes $\psi: {\mathcal G} \rightarrow {\mathcal F}$ that lifts the map of $R$-modules $R/(I_{j-1}:f_j) \rightarrow R/I_{j-1}$.   In the context of the generalized EK resolution described above, this choice was encoded in the definition of the decomposition function $\psi$. We will see that varying the decomposition function leads to combinatorially distinct geometric complexes supporting cellular resolutions, recovering the results from the papers mentioned above.  We view the cellular mapping cone construction as a means of unifying the various constructions of cellular resolutions from the literature.

\subsection{The complex of boxes (homomorphism) resolution} \label{HomRes}

In this section we show how a different choice of decomposition function in the mapping cone construction recovers the cellular resolutions of \cite{Sine}, of \cite{CN} and \cite{NR} (where they are called `complex of boxes' resolutions), and of \cite{DE} (where they are constructed as `homomorphism complex' resolutions).

\medskip

In this context we restrict our attention to \emph{cointerval} ideals, a class of hypergraph edge ideals introduced in \cite{DE} and \cite{MKM} that generalize squarefree strongly stable ideals. Recall that a (regular) \emph{$d$--graph} $H$ on vertex set $[n]$ is a collection of subsets of $[n] = \{1,2,\dots,\}$, each of cardinality $d$.  A $d$-graph $H$ naturally gives rise to a (square-free) monomial ideal $I_H$ by taking generators to be the edges of $H$.  To describe the class of cointerval graphs we need the following notion.

\begin{Definition}
Let $H$ be a $d$--graph and let $v \in V(H) \subseteq \mathbb{Z}$ be some vertex. Then the \emph{$v$--layer} of $H$ is a $(d-1)$--graph on $V\setminus v$ with edge set
\[ \{ v_1v_2\cdots v_{d-1} \mid vv_1v_2\cdots v_{d-1} \in E(H)\textrm{ and }v<v_1,v_2,\ldots ,v_{d-1} \}. \]
\end{Definition}

\begin{Definition}\label{def:intHyp}
The class of \emph{cointerval} $d$--graphs is defined recursively as follows.

Any $1$--graph is cointerval. For $d>1$, a finite regular
$d$--graph $H$ with vertex set $V(H)\subseteq \mathbb{Z}$ is \emph{cointerval}
if
\begin{itemize}
\item[(1)] for every $i \in V(H)$ the $i$--layer of $H$ is cointerval;
\item[(2)] for every pair $i<j$ of vertices, the $j$--layer of $H$ is a subgraph of the $i$--layer of $H$.
\end{itemize}
\end{Definition}

When $d = 2$ the class of cointerval graphs can be seen to coincide with the well-studied \emph{complements of interval graphs} of structural graph theory (hence the name).  One can see that cointerval $d$--graphs generalize the class of (pure) \emph{shifted} simplicial complexes.  Given a cointerval $d$-graph $H$, we will often refer to the associated edge ideal $I_H$ as a \emph{cointerval ideal}.

\medskip

In \cite{MKM} the authors work with a class of ideals they call \emph{generalized Ferrers ideals} which can be seen to coincide with the class of cointerval ideals.   We recall the equivalent definition here.

\begin{Lemma}\label{cointerval}
A squarefree monomial ideal $I$ (generated by monomials of degree $d$) is cointerval if and only if for any monomial $m = x_{i_1}x_{i_2} 
\cdots x_{i_d} \in I$ we also have
\[x_{j_1}x_{j_2} \cdots x_{j_t}x_{i_{t+1}} \cdots x_{i_d} \in I, \] 
\noindent
where $(i_1, i_2, \dots, i_d)$ and $(j_1, j_2, \dots, j_d)$ are such that $j_1 \leq i_1, j_2 \leq i_2, \dots, j_t \leq i_t$ for some $t \leq d$. 

\end{Lemma}

From \cite{MKM} we also take the following result.

\begin{Theorem}
\cite[Theorem 2.5]{MKM}
Let $I_H$ be the edge ideal of a generalized $d$-Ferrers hypergraph $H$, so that $I_H$ is a cointerval ideal.  Then $I_H$ is weakly polymatroidal, and in particular has linear quotients with respect to lexicographic order on its generators.
\end{Theorem}

Next we recall the construction of the polyhedral complex $X_H$ that supports a minimal free resolution of the cointerval ideal associated to the hypergraph $H$.  For subsets $\sigma_i, \sigma_j \subset [n] = \{1,2,\dots, n\}$  we say that $\sigma_i < \sigma_j$ if $x < y$ for all $x \in \sigma_i$ and $y \in \sigma_j$.  We use the notation $\Delta_S$ to denote the simplex with vertex set $S$.

\medskip

\begin{Definition}~\label{def:Xsupp}
Let $H$ be a $d$--graph on vertex set $V(H) = \{v_1, v_2 \dots, v_n \}$. The polyhedral complex $X_H$ is defined to be the
subcomplex of the product
\[ \prod_{i=1}^d \Delta_{V(H)} \] satisfying
\begin{itemize}
\item[(1)] The vertices of $X_H$ are $v_{i_1} \times v_{i_2} \times \dots \times v_{i_d}$, where $v_{i_1}v_{i_2} \cdots v_{i_d}$ is an edge of $H$;

\item[(2)] For $\sigma_i \subseteq V(H)$, the cells $\sigma_1 \times \sigma_2 \times \cdots \times \sigma_d$ satisfy $\sigma_1 < \sigma_2 < \cdots < \sigma_d$.
\end{itemize}
\end{Definition}

Note that for any $d$-graph $H$, the faces of the complex $X_H$ are naturally labeled by monomials.  In particular, the vertices are labeled by monomials corresponding to the edges of $H$ (i.e. the generators of $I_H$), and the higher dimensional faces $F = \sigma_1 \times \sigma_2 \times \cdots \times \sigma_d$ are labeled by
\[ \prod_{i=1}^d \prod_{v_j\in \sigma_i} x_{j}, \]
\noindent
which can be seen as equal to the least common multiple of the monomial labels on the vertices of $F$.

\begin{Remark}
Viewing $H$ as a directed $d$-graph (with orientation on the edges given by the integer labels on the vertices), one can regard $X_H$ as a `space of directed edges' of $H$.   Indeed, if we let $E$ denote the $d$-graph with vertex set $[d]$ consisting of a single edge $\{1,2,\dots,d\}$, then $X_H = \textrm{Hom}(E,H)$, a space of directed graph homomorphisms from $E$ to $H$ analogous to the undirected $\textrm{Hom}$ complexes of \cite{BK}.  This perspective was also employed in \cite{BBK} where the authors study ideals arising from more general (nondegenerate) simplicial homomorphisms.
\end{Remark}

The main result from \cite{DE} is that these complexes support minimal cellular resolutions of cointerval ideals.

\begin{Theorem} \cite[Theorem 4.1]{DE}
\label{thm:resCoInt}
Let $H$ be a cointerval $d$--graph on vertex set $[n]$. Then the polyhedral complex $X_H$ supports a minimal cellular resolution of the edge ideal $I_H$.
\end{Theorem}

One nice thing about the spaces $X_H$ is tdhat the differential maps are so easy to describe. Indeed, if $\sigma_1 \times \sigma_2 \times \cdots \times \sigma_d$ is a cell of $X_H$, we can write $d \big(\sigma_1 \times \sigma_2 \times \cdots \times \sigma_d \big)$ as:

\begin{equation} \label{eq:topcell}
\sum_{\ell = 1} ^d (-1)^{\ell-1} \sum_{j} (-1)^{j + |\sigma_1| + \dots + |\sigma_{\ell-1}|}  \sigma_1 \times \sigma_2 \times \cdots \times \{\sigma_{\ell_1}, \sigma_{\ell_2}, \dots, \widehat{\sigma_{\ell_j}}, \dots, \sigma_{\ell_k} \} \times \cdots \times \sigma_d. 
\end{equation}

\noindent
Here we remove the element $\sigma_{i_j}$ only if it leaves a nonempty subset, that is if $|\sigma_i| \geq 2$.  These are the differential maps we wish to recover in Theorem \ref{Theorem_Homcone}, where we show that the homomorphism complex $X_H$ can in fact also be realized as an iterated geometric mapping cone construction.  We first need another description of the basis elements.

\begin{Lemma} \label{lemma:basis}
Let $I = I_H$ be a cointerval ideal associated to the cointerval $d$--graph $H$, and let $F$ be its minimal free resolution.  Then the basis elements for each free $R$-module $F_i$ determined by the cellular resolution $X_H$ correspond to the symbols
\[ (m;\alpha),\quad {\textrm where} \; \; m\in G(I), \quad \alpha\subset \set(m),\quad |\alpha|=i-1. \]
\end{Lemma}

\begin{proof}
Suppose $I_H$ is a cointerval ideal with monomial generators $m_1, m_2, \dots, m_k$ listed in lexicographical order.  From Lemma \ref{cointerval} we see that for any generator $m = x_{i_1}x_{i_2} \cdots x_{i_d}$ we have
\[\set(m) = \{j \in [n]: \textrm{$j < i_1$\, or for some $t \geq 1$ we have $i_t < j < i_{t+1}, \; x_{i_1} \cdots x_{i_t} x_j x_{i_{t+2}} \cdots x_{i_d} \in I\}$}. \]

Since $X_H$ supports a minimal cellular resolution of $I_H$, we have that a basis for $F_i$ is given by the number of $i-1$ dimensional faces of $X_H$.
 From Definition \ref{def:Xsupp} we have an explicit description of these faces, which we now want to show are naturally labeled by the symbols $(m; \alpha)$.

Suppose $m = x_{i_1}x_{i_2} \cdots x_{i_d} \in G(I)$ and $\alpha \subset \set(m)$.  We associate the symbol $(m; \alpha)$ to the face $\sigma_1 \times \sigma_2 \times \cdots \times \sigma_{d}$, where for $1 \leq \ell \leq d$ we define $\sigma_\ell = \{i_\ell\} \cup \{j \in \alpha: i_{\ell-1} < j < i_\ell\}$ (by convention we set $i_{0} = 0$).  To see that this assignment defines a bijection we describe the inverse.  For this, suppose $\sigma_1 \times \sigma_2 \times \cdots \times \sigma_d$ is a face of $X_H$.  Define $m = x_{i_1}x_{i_2} \cdots x_{i_d}$, where for each $\ell$ with $1 \leq \ell \leq d$ we set $i_\ell = \max(\sigma_\ell)$.  Then define 

\[ \alpha = \bigcup_\ell \;  (\sigma_\ell \backslash i_\ell). \]

\noindent
One can check that these assignments are inverses of one another.
\end{proof}

The differentials of the resolution supported on the homomorphism complex are described by the incidence face structure of the polyhedral complex $X_H$.  To recover these differentials as an iterated mapping cone, we need a new notion of decomposition function.

\medskip

Let $I$ be a cointerval monomial ideal with the sequence of generators $m_1, m_2, \ldots,m_k$ listed in lexicographical order.  As above let $G(I)$ be the set of all monomials in $I$.  Let $xG(I)$ be the set of all products $x_i m_j$, where $m_j \in G(I)$ is a generator of $I$, and $x_i$ is a variable in the polynomial ring such that $i \in \set(m_j)$. 

\begin{Definition}

With the notation established a above, the map $c: xG(I)\rightarrow G(I)$ is defined as follows.   Suppose $m_j = x_{j_1} x_{j_2} \cdots x_{j_d}$, with $j_1 < j_2 < \cdots < j_d$ and let $x_i$ such that $i \in \set(m)$.   Pick the minimum $j_{k}$ such that $j_{k-1} < i \leq j_{k}$, and define $c(x_i m) = x_i m/x_{j_k}$.

% If there exists another $i^\prime \in \set(m)$ such that $j_{k-1} < i < i^\prime \leq j_k$ define $c(x_i m) = 0$.  
\end{Definition}

We note that this assignment is different than the \emph{decomposition function} from \cite{HT}.  In particular, our assignment does \emph{not} in general satisfy 
\[c(x_sc(x_t m)) = c(x_tc(x_sm)) \]

(the edge ideal of the complete 3-graph on $n = 5$ vertices provides a counterexample, where $\set(x_3x_4x_5) = \{1,2\}$). 

We will use this new decomposition function $c$ to encode the map of chain complexes $\psi$ involved in the mapping cone construction.  We first set up some further notation.  For any generator $m=x_{i_1}x_{i_2}\cdots x_{i_d}$ of $I$ we decompose $\set(m)$ as 
\[
\set(m)=A_1\cup A_2\cup\cdots \cup A_d,
\]
where 
\[
A_\ell=\{j\in [n]:\  i_{\ell-1}< j < i_\ell {\rm\ such\ that\ we\ have\ } x_{i_1} \cdots x_{i_{\ell-1}} x_j x_{i_{\ell+1}} \cdots x_{i_d} \in I\}.
\]

\begin{Notation}
Let $\alpha\subset \set(m)$. For each $1\leq \ell\leq d$, let $s_\ell$ denote the maximum element of $\alpha\cap A_\ell$, and $s'_{\ell}$ denote the maximum element of $\alpha\cap(A_\ell\backslash\{s_\ell\})$. Let $T(\alpha)=\{s_1,\ldots,s_d\}$.
\end{Notation}

\medskip

\begin{Example}
Consider again our running example 
\[I = \langle x_1x_2, x_1x_3,x_1x_5, x_2x_3, x_2x_5, x_3x_5, x_4x_5 \rangle. \]
\noindent
We see that for $m = x_2x_5$ and $\alpha = \set(x_2x_5) = \{1,3\}$, we have $T(\alpha) = \{1,3\}$.  On the other hand, for $m = x_3x_5$ and $\alpha = \set(x_3 x_5) = \{1,2\}$, we get $T(\alpha) = \{2\}$.
\end{Example}

\begin{Remark}\label{rem:s,t}
Let $t,s\in A_\ell$ for some $\ell$ and $t>s$. Then $c(x_s(x_tm))=c(x_sm)$, since $i_{\ell-1}<s<t<j_\ell$.  Here we again employ the shorthand notation $c(x_s(x_tm)) = c(x_s c(x_tm))$.
\end{Remark}

\begin{Remark}
Note that \eqref{eq:topcell} can be written as 
\begin{eqnarray*}
\sum_{\ell = 1} ^d (-1)^{\ell -1} \sum_{j=1}^{s_\ell} (-1)^{j + |\sigma_1| + \cdots + |\sigma_{\ell - 1}|} \sigma_1 \times \sigma_2 \times \cdots \times \{\sigma_{\ell_1}, \sigma_{\ell_2}, \dots, \widehat{\sigma_{\ell_j}}, \dots, \sigma_{\ell_k} \} \times \cdots \times \sigma_d
\\
+ \sum_{\ell = 1\atop {|A_\ell|>0}} ^d (-1)^{\ell - 1 +|\sigma_1| + \cdots + |\sigma_{\ell}|} \sigma_1 \times \sigma_2 \times \cdots \times \sigma_\ell\backslash \max(\sigma_{\ell})  \times \cdots \times \sigma_d.
\end{eqnarray*}
Recall from Lemma \ref{lemma:basis} that each $\sigma_1 \times \sigma_2 \times \cdots \times \sigma_d$ corresponds to a symbol $(m; \alpha)$.  The first summand above is taken over all elements of $\alpha$ and the second summand is over the indices of the elements of $\supp(m)$.  The second summand can be also considered over the elements of $T(\alpha)$ since in case that $|T(\alpha)\cap A_\ell|>0$ we have that $c(x_{s_\ell}m)={x_{s_\ell} m}/x_{\max(\sigma_{\ell})}$. 
\end{Remark}

\begin{figure}
\begin{center}
  \includegraphics[scale = .85]{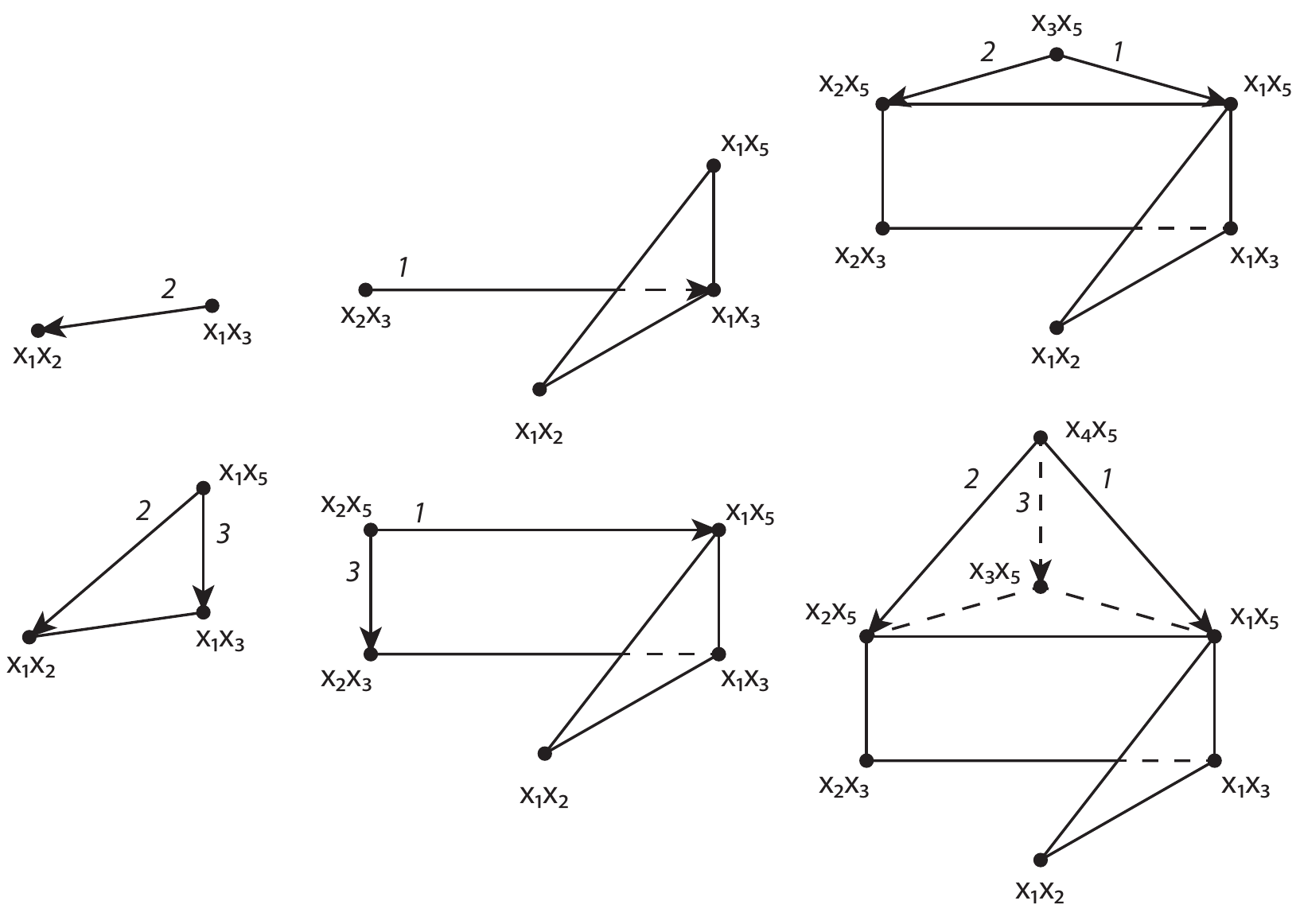}
\label{Homresolution}
    \caption{The resolution of $I = \langle x_1x_2, x_1x_3,x_1x_5, x_2x_3, x_2x_5, x_3x_5, x_4x_5 \rangle$ built from an iterated mapping cone as a realization of the homomorphism complex.}

\end{center}
\end{figure}

\begin{Theorem} \label{Theorem_Homcone}
Let $I = I_H$ be the monomial edge ideal associated to a cointerval $d$-graph $H$, and ${\mathcal F}$ the graded
minimal free resolution of $R/I$ obtained from the homomorphism complex $X_H$.
 Then ${\mathcal F}$ is realized as an iterated mapping cone, with the basis for each module $F_i$ as above.  The chain map $d$ of ${\mathcal F}$ is given by
 \[
     d(m;\alpha) = \sum_{j_i\in \alpha}(-1)^{i} x_{j_i} (m;\alpha\setminus {j_i}) +
\sum_{{j_i\in T(\alpha)}}(-1)^{i-1} \frac{x_{j_i} m}{c(x_{j_i}m)} (c(x_{j_i} m);\alpha\setminus {j_i})
\]

if $\sigma\neq \emptyset$, where $\alpha=\{{j_1},\ldots, {j_p}\}\subset \set(m)$ with $j_1<\cdots<j_p$, and
$d(m,\emptyset) = m$ otherwise.

 \end{Theorem}

\begin{proof}
We follow the strategy of the proof of Theorem 1.12 in \cite{HT}.  Let $I = I_H$ be a cointerval interval ideal associated to the cointerval $d$-graph $H$, and let $X_H$ be its homomorphism complex.  We show by induction on $j$ that the complex $F^{(j)}$ has the desired boundary map.  By definition $F^{(j+1)}$ is the mapping cone of $\psi^{(j)}: K^{(j)} \rightarrow F^{(j)}$ so that $F^{(j)}$ is a subcomplex of $F^{(j+1)}$ and hence it is enough to check the formula on the basis elements $(m, \alpha)$.

The definition of the mapping cone differential gives $d(m_{j+1};\alpha) = -d_K(m_{j+1};\alpha) + \psi^{(j)}(m_{j+1};\alpha)$, where $d_K$ is the differential of the relevant Koszul complex $K^{(j)}$.
 
Hence it is enough to show that we can define $\psi^{(j)}$ according to
\[\psi^{(j)} (m_{j+1};\alpha) = \sum_{{j_i\in T(\alpha)}}(-1)^{i-1} \frac{x_{j_i} m_{j+1}}{c(x_{j_i}m_{j+1})} (c(x_{j_i} m_{j+1});\alpha\setminus {j_i}),\]
\noindent
if $\alpha \neq \emptyset$, and otherwise $\psi^{(j)}(m_{j+1};\emptyset) = m_{j+1}$.

For this we must verify that $\psi^{(j)} \circ d_K = d \circ \psi^{(j)}$. To simplify notation we let $m = m_{j+1}$ and $\psi = \psi_{j+1}$.

For $t \in \set(m)$ a singleton element we have

\[ \big(\psi \circ d_K \big) \big( (m; \{t\}) \big) \; = \; \psi \big(x_t(m; \emptyset) \big) \; = x_t m,\]

\noindent
and on the other hand

\[ \big(d \circ \psi \big) \big((m; \{t\}) \big) \; = \; d \big( \frac{x_t m}{c(x_t m)} (c(x_t m); \emptyset) \big) \\
\; = \frac{x_t m}{c(x_tm)} c(x_t m) = x_t m. \]

\noindent
For larger subsets of $\set(m)$, next consider $\sigma \subset \set(m)$ with $|\sigma| \geq 2$.  In this case we have

\begin{equation} \label{psid}
\begin{aligned}
\big(\psi \circ d_k \big) \big( (m; \alpha) \big)&=\sum_{t \in \alpha} (-1)^{\epsilon(\alpha; t)} x_t \psi \big( (m; \alpha \backslash t) \big) \\
&=\sum_{t \in \alpha} (-1)^{\epsilon(\alpha;t)} x_t \Large( \sum_{s \in {T(\alpha \backslash t)}} (-1)^{\epsilon(\alpha \backslash t; s)} \frac{x_s m}{c(x_s m)} (c(x_s m); \alpha \backslash \{s,t\})\\
&=\sum_{t \in \alpha} \sum_{s \in {T(\alpha \backslash t)}}  (-1)^{\epsilon(\alpha; t) + \epsilon(\alpha \backslash t;s)} \frac{x_t x_s m}{c(x_s m)} (c(x_s m); \alpha \backslash \{s,t\}).
\end{aligned}
\end{equation}

\noindent
Here we use the notation $\epsilon(\alpha;t) = |\{s \in \alpha \; : \; s < t\}|$.

\medskip
The other composition gives us

\begin{equation} \label{dpsi}
\begin{aligned}
\big(d \circ \psi)\big((m;\alpha) \big)&=\sum_{t \in {T(\alpha)}} (-1)^{\epsilon(\alpha;t)} \frac{x_t m}{c(x_t m)} d\big( (c(x_t m); \alpha \backslash t) \big),
\end{aligned}
\end{equation}

\noindent
where

\[
\begin{aligned}
d \big( (c(x_t m); \alpha \backslash t) \big) &=\sum_{s \in \alpha \backslash t} (-1)^{\epsilon(\alpha \backslash t; s)} x_s (c(x_t m); \alpha \backslash \{s,t\}) \\
&+\sum_{s \in {T(\alpha \backslash t)}} (-1)^{\epsilon(\alpha \backslash t; s)} \frac{x_s c(x_t m)}{c(x_s c(x_t m))} \big(c(x_s c(x_t m)); \alpha \backslash \{s,t\}\big).
\end{aligned}
\]

\noindent
Therefore Equation \ref{dpsi} becomes 

\begin{equation} \label{new_dpsi}
\begin{aligned}
\sum_{t \in {T(\alpha)}} (-1)^{\epsilon(\alpha;t)} \frac{x_t m}{c(x_t m)}\sum_{s\in T(\alpha\backslash\{t\})}  (-1)^{\epsilon(\alpha \backslash t; s)} \frac{x_s c(x_tm)}{c(x_s c(x_t m))} \big( c(x_s c(x_t m)); \alpha \backslash \{t,s\}\big)\\
+
\sum_{t \in {T(\alpha)}} (-1)^{\epsilon(\alpha;t)} \frac{x_t m}{c(x_t m)}\sum_{s\in \alpha\backslash\{t\}}  (-1)^{\epsilon(\alpha \backslash t; s)} x_s \big( c(x_t m); \alpha \backslash \{t,s\} \big).\\
\end{aligned}
\end{equation}

\medskip

\noindent
Now note that the first summand expressing $\big(d \circ \psi)\big((m;\alpha) \big)$ in Equation \eqref{new_dpsi} can be written
\[
\begin{aligned}
\sum_{t \in {T(\alpha)}} \sum_{s\in T(\alpha\backslash\{t\})}  (-1)^{\epsilon(\alpha; t) + \epsilon(\alpha \backslash t;s)} \frac{x_s x_t m}{c(x_s c(x_t m))} \big( c(x_s c(x_t m)); \alpha \backslash \{t,s\}\big),\\
\end{aligned}
\]

\noindent
which  can be expanded as

\begin{equation}\label{expanded}
%\[
\begin{aligned}
\sum_{\ell=1}^{d}\sum_{t=s_\ell} \sum_{s=s_{k}\atop k \neq \ell }  (-1)^{\epsilon(\alpha; t) + \epsilon(\alpha \backslash t;s)} \frac{x_s x_t m}{c(x_s c(x_t m))} \big( c(x_s c(x_t m)); \alpha \backslash \{t,s\} \big)
&\\+
\sum_{\ell=1}^{d}\sum_{t=s_\ell} \sum_{s=s'_{\ell}}  (-1)^{\epsilon(\alpha; t) + \epsilon(\alpha \backslash t;s)} \frac{x_s x_t m}{c(x_s c(x_t m))} \big(c(x_s c(x_t m)); \alpha \backslash \{t,s\}\big).\\
\end{aligned}
%\]
\end{equation}

\noindent
Exchanging the role of $\ell$ and $k$ in the first summand of \ref{expanded} shows that it is `zero'.

The second summand of \ref{new_dpsi} can be written
\[
\begin{aligned}
\sum_{t \in {T(\alpha)}} \sum_{s\in \alpha\backslash\{t\}} 
(-1)^{\epsilon(\alpha; t) + \epsilon(\alpha \backslash t;s)} \frac{x_t x_sm}{c(x_t m)}  \big(c(x_t m); \alpha \backslash \{t,s\}\big),\\
\end{aligned}
\]

\medskip

\noindent
and after applying the observation from Remark~\ref{rem:s,t} we obtain
\begin{equation}\label{newest_dpsi}
\begin{aligned}
\big(d \circ \psi)\big((m;\alpha) \big)&=\sum_{\ell=1}^{d}\sum_{t=s_\ell} \sum_{s=s'_{\ell}}  (-1)^{\epsilon(\alpha; t) + \epsilon(\alpha \backslash t;s)} \frac{x_s x_t m}{c(x_s m)} \big(c(x_s m); \alpha \backslash \{t,s\} \big) \\
%\sum_{\ell=1}^{d}\sum_{t=s_\ell} \sum_{s=s'_{\ell}}  (-1)^{\epsilon(\alpha; t) + \epsilon(\alpha \backslash t;s)} \frac{x_s x_t m}{c(x_s(c(x_t m)))} \big( (c(x_s(c(x_t m))); \alpha \backslash \{t,s\}\big)\\
&+
\sum_{t \in {T(\alpha)}} \sum_{s\in \alpha\backslash\{t\}} 
(-1)^{\epsilon(\alpha; t) + \epsilon(\alpha \backslash t;s)} \frac{x_t x_sm}{c(x_t m)}  \big( c(x_t m); \alpha \backslash \{t,s\} \big).\\
\end{aligned}
\end{equation}

%\noindent
%By Remark~\ref{rem:s,t} the first summand of \ref{newest_dpsi} is 
%\[
%\sum_{\ell=1}^{d}\sum_{t=s_\ell} \sum_{s=s'_{\ell}}  (-1)^{\epsilon(\alpha; t) + \epsilon(\alpha \backslash t;s)} \frac{x_s x_t m}{c(x_s m)} \big(c(x_s m); \alpha \backslash \{t,s\}) \big).
%\]

Next we compare the indices appearing in the non-zero summands corresponding to Equation \ref{psid} for $\big(\psi\circ d_k)\big((m;\alpha) \big)$ and Equation \ref{newest_dpsi} for $\big(d \circ \psi)\big((m;\alpha) \big)$,  The indices appearing in the non-zero summands of $\big(\psi\circ d_k)\big((m;\alpha) \big)$ consist of 
\[
\begin{aligned}
 &\big \{(t,s_{k}) :\ t \in A_{\ell}, \; \ell \neq k \big \} \; \bigcup \; \big \{(t,s_{\ell}) :\ t \in A_{\ell}  \backslash s_{\ell} \big \} \; \bigcup \; \big \{(s_{\ell},s^\prime_{\ell}) \}. \\
\end{aligned}
\]

\medskip

Recall our notation, here $s_{k}$ denotes the maximum element of $\alpha \cap A_{k}$, and $s^\prime_{\ell}$ denotes the maximum element of $\alpha \cap A_{\ell} \backslash s_{\ell}$.

On the other hand the indices appearing in the non-zero summands of $\big(d \circ \psi)\big((m;\alpha) \big)$ consist of 
\[
\begin{aligned}
 & \big \{(s_{\ell},s^\prime_{\ell}) \big \} \; \bigcup \; \big \{ (s_{k}, j) :\ j \in A_{\ell}, \; \ell \neq k \big \} \; \bigcup \;  \big \{(s_{\ell}, j) :\ j \in A_{\ell}  \backslash s_{\ell} \big \} 
 . \\
\end{aligned}
\]

\noindent
Exchanging the roles of $s$ and $t$ completes the proof.
\end{proof}

\medskip

%We note that in the case that $I$ is taken to be a power of the maximal ideal, this resolution coincides with the \emph{staircase resolution} described in \cite{DJS}.

%\textcolor{blue}{I added the observation in the following Remark:}

\begin{Remark}
The cells that serve as basis elements of the free modules in the resolution constructed in Theorem~\ref{Theorem_Homcone} can also be obtained
 by gluing together the simplices  ${\rm ch}(m,\alpha,\sigma)$ corresponding to the different choices of permutations $\sigma$ of $\alpha$ from the set $\mathcal{A}$, where 
\[\mathcal{A}=\{\sigma=(\sigma_1,\ldots,\sigma_p):\ \ {\rm if\ } \sigma_i,\sigma_j\in A_\ell {\rm \ for\ some\ }\ell \ {\rm and\ }\sigma_i>\sigma_j\ {\rm then\ } i<j\}.\]  

For this we define the cell $U(m, \alpha)$ as the union over all permutations $\sigma\in \mathcal{A}$ of $\alpha$.

\[ U(m, \alpha) = \bigcup_{\sigma\in\mathcal{A}} {\rm ch}(m,\alpha,\sigma) .\]
\end{Remark}

\begin{comment}
\begin{Lemma}
Let $\alpha\subset \set(m)$. Then
%Let $A(m)$ be the set of the monomials appearing in $U(m,\alpha)$. Then  
\[
U(m,\alpha)=\sigma_1\times \sigma_2\times\cdots\times \sigma_d,
\]
where 
\[
\sigma_1=A_1\cup\{i_1\},\ \sigma_2=A_2\cup\{i_2\},\ldots, \sigma_d=A_d\cup\{i_d\}.
\]
\end{Lemma}
\textcolor{blue}{Do we need the above construction of cells and the lemma?}
\end{comment}

%\subsection{The NPS resolution}
%In \cite{NPS} the authors consider ideals that arise from oriented matroids and matroids, and construct minimal resolutions.  In the case that the (oriented) matroid arises from a hyperplane arrangement in Euclidean space, the resolution is supported on the complex of bounded cells  associated to the arrangement. 
%QUESTION: Is this cellular resolution an instance of the Mermin type construction? In small examples this seems to  be the case.

\subsection{Other regular decomposition functions}\label{RegRes}
Describing all possible cellular realizations of the mapping cone resolution of an ideal $I$ with linear quotients seems to be difficult (even for a fixed ordering of the generators).  Perhaps a more manageable task would be to restrict to those cellular resolutions one obtains from regular decomposition functions.

\medskip

For an ideal $I$ with linear quotients and a regular decomposition function, Herzog and Takayama in \cite[Theorem 1.12]{HT} provide an explicit description of the differential map of the resolution of an ideal $I$ obtained from an iterated mapping cone.   One can check that the proof of the theorem (and the related lemmas) relies only on the regularity of the decomposition function $b$, and not the definition of $b$ in terms of assigning a particular generator of $I$ to every monomial.  Similarly, our cellular realization in Theorem \ref{Theorem_cone} only relies on the regularity property.  Hence we can vary the decomposition function in each step and apply those results to obtain combinatorially distinct geometric complexes that support resolutions, as the next example illustrates.

\begin{Example}
Consider the ideal given by $I = \langle x_1x_2, x_1x_3,x_1x_5, x_2x_3, x_2x_5, x_3x_5, x_4x_5 \rangle$.  One can check that $I$ has linear quotients with respect to this ordering.  We construct a resolution of $I$ via the iterated mapping cone procedure, and by choosing different regular decomposition functions at each step we arrive at a family of combinatorially distinct complexes supporting the resolution.

\begin{figure}
\begin{center}
  \includegraphics[scale = .8]{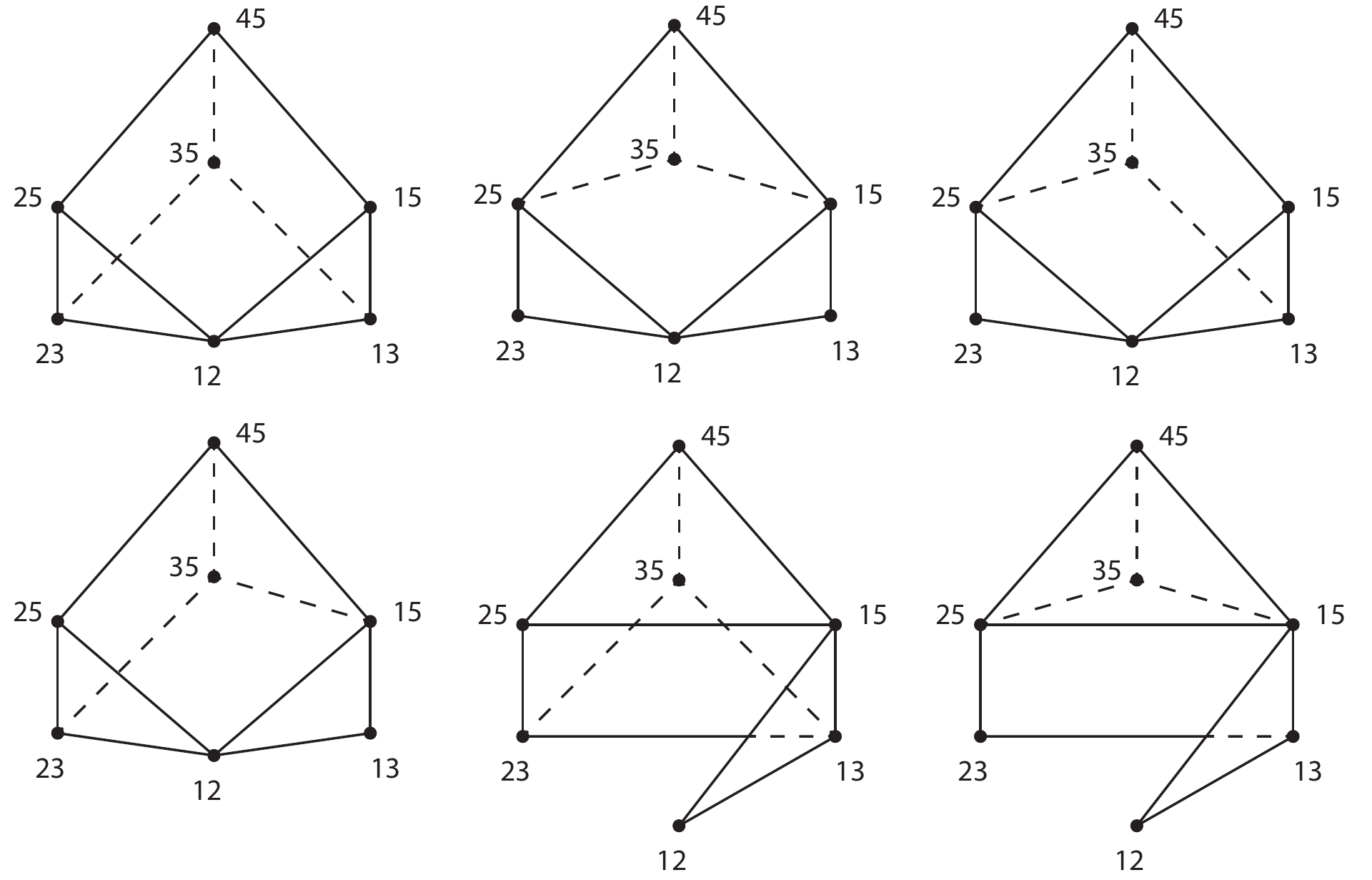}
\label{resolutions_all}
    \caption{The family of cellular resolutions obtained by considering all possible regular decompositions of the ideal $I$ with linear quotient order $(x_1x_2, x_1x_3,x_1x_5, x_2x_3, x_2x_5, x_3x_5, x_4x_5)$.  Here we use the notation $ij$ to denote the variable $x_ix_j$.  The first complex is the generalized EK resolution from Section \ref{EKres} and the last is the homomorphism complex of Section 
   \ref{MoreRes}.}

\end{center}
\end{figure}

\end{Example}

\subsection{Further questions - a space of resolutions?}\label{ResSpace}

In each step of the mapping cone construction we have a choice of homomorphism of complexes $\psi: {\mathcal G} \rightarrow {\mathcal F}$ that lifts the map of $R$-modules $R/(I_{j-1}:f_j) \rightarrow R/I_{j-1}$.  As we have seen, this choice of homomorphism can be encoded in the cellular structure of the Koszul simplex that is glued onto the previously constructed resolution.  After fixing a basis the collection of all such choices of bluings forms a finite set, but can we understand them as comprising some geometric object and hence obtain a `space of mapping cone resolutions'?

\medskip

We note that if $I = \langle x_1, x_2, \dots, x_n \rangle^d$ is a power of the graded maximal ideal, a certain space of cellular resolutions of $I$ is in fact described in \cite{DJS}.  In this context a cellular resolution of $I$ is obtained by a generic arrangement of tropical hyperplanes, which in turn corresponds to a regular triangulation of the product of simplices $\Delta_{n-1} \times \Delta_{d-1}$.  The collection of all regular triangulations of $\Delta_{n-1} \times \Delta_{d-1}$ (or any polytope) has a natural polyhedral structure known as a \emph{secondary polytope}.

\medskip
A natural question arrises: if we fix the linear quotient order on the generators of an ideal $I$, what are the possible combinatorial types of complexes that we see as we attach a simplex in each step of the cellular mapping cone construction?  How many different choices do we have to glue in the simplex?  At one extreme sits the maximal ideal $I = \langle x_1, x_2, \dots x_n \rangle$, where we have no choice but to build another simplex of one more dimension in each step.  The space of resolutions in this case is a single point.  As seen in \cite{DJS}, already for the square of the maximal ideal in 4 variables we see distinct complexes arising.  

%%%%%%%%%%%%%%%%%%%%%%%%%%%%%%%%%%%%%%%%%%%%%%%%%%%%%%%%%%%%%%%%%%%%%%%%%%%%%%%%%%%%%%%%%%%%%%%%%%%%%%%%%%%%%%%

\bibliography{cone.bib}
\bibliographystyle{plainnat}

\end{document}